\documentclass[preprint,11pt]{elsarticle}
\makeatletter
\def\ps@pprintTitle{%
 \let\@oddhead\@empty
 \let\@evenhead\@empty
 \def\@oddfoot{\centerline{\thepage}}%
 \let\@evenfoot\@oddfoot}
\makeatother
\usepackage{amssymb,amsmath,amsthm}
\usepackage{epsfig}
\usepackage{color}
\usepackage{makeidx}
\usepackage{bbm}
\usepackage[top=1in, bottom=1in, inner=1in, outer=1in]{geometry}
\usepackage{setspace}

\newtheorem{theorem}{Theorem}

\newtheorem*{propa}{Proposition A}
\newtheorem*{propb}{Proposition B}
\newtheorem*{propc}{Proposition C}
\newtheorem*{propd}{Proposition D}
\newtheorem{definition}{Definition}
\newtheorem{lemma}{Lemma}

\newtheorem{proposition}{Proposition}

\newtheorem{remark}{Remark}
\newtheorem{conjecture}{Conjecture}
\newcommand*\xbar[1]{%
  \hbox{%
    \vbox{%
      \hrule height 0.5pt 
      \kern0.4ex
      \hbox{%
        \kern-0.15em
        \ensuremath{#1}%
        \kern-0.15em
      }%
    }%
  }%
}

\begin{document}
\begin{frontmatter}

\title{Branching Brownian motion in an expanding ball and application to the mild obstacle problem}

\author{Mehmet \"{O}z}
\ead{mehmet.oz@ozyegin.edu.tr}
\ead[url]{https://faculty.ozyegin.edu.tr/mehmetoz/}

\address{Department of Natural and Mathematical Sciences, Faculty of Engineering, \"{O}zye\u{g}in University, Istanbul, Turkey}

\begin{abstract}

We first study a $d$-dimensional branching Brownian motion (BBM) among mild Poissonian obstacles, where a random \emph{trap field} in $\mathbb{R}^d$ is created via a Poisson point process. The trap field consists of balls of fixed radius centered at the atoms of the Poisson point process. The mild obstacle rule is that when particles are inside traps, they branch at a lower rate, which is allowed to be zero, whereas when outside traps they branch at the normal rate. We prove upper bounds on the large-deviation probabilities for the total mass of BBM among mild obstacles, which we then use along with the Borel-Cantelli lemma to prove the corresponding strong law of large numbers. Our results are quenched, that is, they hold in almost every environment with respect to the Poisson point process. Our strong law improves on the existing corresponding weak law in \cite{E2008}. Then, we study a $d$-dimensional BBM inside subdiffusively expanding balls, where the boundary of the ball is \emph{deactivating} in the sense that once a particle of the BBM hits the moving boundary, it is instantly deactivated but can be reactivated at a later time provided its ancestral line is fully inside the expanding ball at that later time. We obtain a large-deviation result as time tends to infinity on the probability that the mass inside the ball is aytpically small. An essential ingredient in the proofs of the mild obstacle problem turns out to be the large-deviation result on the mass of BBM inside expanding balls.
 
\end{abstract}

\vspace{3mm}

\begin{keyword}
Branching Brownian motion \sep killing boundary \sep large deviations \sep strong law of large numbers \sep Poissonian traps \sep random environment
\vspace{3mm}
\MSC[2010] 60J80 \sep 60K37 \sep 60F15 \sep 60F10 
\end{keyword}

\end{frontmatter}

\pagestyle{myheadings}
\markright{BBM in expanding ball and among mild obstacles\hfill}

\section{Introduction}\label{intro}

In this work, we study two models involving a branching Brownian motion (BBM), where the population growth is slower than that of an ordinary (free) BBM. It is well-known that typically the population, that is, the total mass, of a BBM grows exponentially in time. To be precise, if $N_t$ denotes the total mass of a strictly dyadic BBM at time $t$ and $\beta$ is the branching rate, then $(N_t)_{t\geq 0}$ is a Yule process, and the limit $M:=\lim_{t\rightarrow\infty}N_t\, e^{-\beta t}$ exists, is positive and finite a.s. In each of the two models considered in this paper, there is a reproduction suppressing mechanism which contributes a subexponentially decaying factor to the typical population size. 

We first consider the model of BBM among \emph{mild} Poissonian obstacles. We study the growth of mass of a BBM evolving in a random environment in $\mathbb{R}^d$, which is composed of randomly located spherical \emph{traps} of fixed radius with centers given by a Poisson point process (PPP). The mild obstacle rule is that when a particle is inside the traps, it branches at a lower rate, which is allowed to be zero, than usual, that is, when it is not in a trap. The mild obstacle problem for BBM was proposed by Engl\"ander in \cite{E2008}, and on a set of full measure with respect to the PPP, a kind of weak law of large numbers (WLLN) was obtained (see \cite[Thm.\ 1]{E2008}) for the mass of the process as well as a result on its spatial spread (see \cite[Thm.\ 2]{E2008}). In Theorem~\ref{thm2}, by estimating successive large-deviation probabilities, we improve the WLLN in \cite{E2008} to the strong law of large numbers (SLLN). We also include the possibility of no branching inside the traps, which was not covered in \cite{E2008}. The case of zero branching inside the traps adds a major challenge to the problem. In this case, it is difficult to show that exponentially many particles are produced with `high' probability in the presence of mild obstacles, whereas if the branching inside the traps is a positive constant, say $\beta_1$, then a growth of $\sim e^{\beta_1 t}$ particles `comes for free' since the branching rate is at least $\beta_1$ everywhere on $\mathbb{R}^d$. 

Then, we consider a BBM in an expanding ball of fixed center, where the radius of the ball is increasing subdiffusively in time, and the boundary of the ball is \emph{deactivating} in the sense that once a particle of the BBM hits the moving boundary, it is instantly deactivated but can be reactivated at a later time provided its ancestral line is fully inside the expanding ball at that later time. On this model, in Theorem~\ref{thm1}, we obtain a large-deviation result, giving the large-time asymptotic behavior of the probability that the total mass of the BBM is atypically small in the expanding ball. 

An essential ingredient in the proof of Theorem~\ref{thm2} turns out be Theorem~\ref{thm1}, that is, the lower tail asymptotics for the mass of BBM in expanding balls. In the mild obstacle problem, a suitable time-dependent \emph{clearing} (see Definition~\ref{def1}) in the random environment in $\mathbb{R}^d$ serves as the expanding ball of Theorem~\ref{thm1}.  

\subsection{Formulation of the problems} 

We now describe the two sources of randomness in the models, and formulate the problems in a precise way. 

\textbf{1. Trap field and mild obstacle problem for BBM:} The setting of random obstacles in $\mathbb{R}^d$ is formed as follows. Let $\Pi$ be a Poisson point process (PPP) in $\mathbb{R}^d$ with constant intensity $\nu>0$, and $(\Omega,\mathbb{P})$ be the corresponding probability space with expectation $\mathbb{E}$. By a \emph{trap} associated to a point $x\in\mathbb{R}^d$, we mean a closed ball of fixed radius $a>0$ centered at $x$, and by a \emph{trap field}, we mean the random set
\begin{equation}
K=K(\omega):=\bigcup_{x_i\in\,\text{supp}(\Pi)}\bar{B}(x_i,a), \label{eqtrapfield}
\end{equation}
where $\bar{B}(x,a)$ denotes the closed ball of radius $a$ centered at $x\in\mathbb{R}^d$. 

In the first part of the current work, a branching Brownian motion (which is briefly described below) is assumed to live in $\mathbb{R}^d$, to which $K$ is attached. For $\omega\in\Omega$, we refer to $\mathbb{R}^d$ with $K(\omega)$ attached simply as the random environment $\omega$, and use $P^\omega$ to denote the conditional law of the BBM in the random environment $\omega$. The mild obstacle problem for BBM has the following rule: when a particle of BBM is in $K^c$, it branches at rate $\beta_2$, whereas when in $K$, it branches at a lower rate $\beta_1$ with $0\leq \beta_1<\beta_2$. That is, under the law $P^\omega$, the BBM has a spatially dependent branching rate
$$  \beta(x,\omega):=\beta_2\,\mathbbm{1}_{K^c(\omega)}(x)+\beta_1\,\mathbbm{1}_{K(\omega)}(x).$$
Our focus is on the total mass of BBM among mild obstacles. We first find upper bounds that are valid for large $t$ on the LD probabilities that the mass is atypically small and atypically large. Then, via Borel-Cantelli arguments, we obtain the corresponding SLLN, which says that the total mass of BBM among mild obstacles grows as its expectation as $t\to\infty$. That is, on a set of full $\mathbb{P}$-measure, in a loose sense, 
$$ \frac{\log N_t}{t} \:\approx\:  \frac{\log E^\omega[N_t]}{t} , \quad t\to\infty, \quad P^\omega\text{-a.s.} $$ 
(See Theorem~\ref{thm2} and Remark~\ref{remark4} for precise statements.) The result is valid in almost every environment $\omega$; hence, it is called a \emph{quenched} SLLN. We refer the reader to \cite[Section 1.2]{E2008} for a list of problems that serve as motivation to study the current model.

We emphasize that here we allow the possibility of complete suppression of branching in $K$, that is, $\beta_1=0$, which was not considered in \cite{E2008}. The case $\beta_1=0$ is inherently harder to deal with. The challenge is to produce exponentially many particles with `high' probability in the presence of mild traps. If $\beta_1=0$, it is difficult to show this, because there are random regions in $\mathbb{R}^d$ where particles don't branch at all; whereas if the branching inside the traps is a positive constant, say $\beta_1$, then a growth of $\sim e^{\beta_1 t}$ particles `comes for free' since the branching rate is at least $\beta_1$ everywhere on $\mathbb{R}^d$. Hence, if $\beta_1=0$, one must have some control over the motion of the BBM inside the traps. A priori, one may think that with significant probability the particles spend too much time inside the traps so that exponential growth cannot be achieved. We show, inter alia, in the proof of the lower bound of Theorem~\ref{thm2} that this is not the case.

\textbf{2. Branching Brownian motion in an expanding ball:} Let $Z=(Z_t)_{t\geq 0}$ be a strictly dyadic $d$-dimensional BBM with branching rate $\beta>0$, where $t$ represents time. Strictly dyadic means that every time a particle branches it gives precisely two offspring. The process can be described as follows. It starts with a single particle, which performs a Brownian motion (BM) in $\mathbb{R}^d$ for a random lifetime, at the end of which it dies and simultaneously gives birth to two offspring. Similarly, starting from the position where their parent dies, each offspring particle repeats the same procedure as their parent independently of others and the parent, and the process evolves through time in this way. All particle lifetimes are exponentially distributed with constant parameter $\beta>0$. For each $t\geq 0$, $Z_t$ can be viewed as a finite discrete measure on $\mathbb{R}^d$, which is supported at the positions of the particles at time $t$. For $t\geq 0$, we use $|Z_t|$ to denote the total mass of $Z$ at time $t$, and occasionally use $N_t:=|Z_t|$. Also, for a Borel set $A\subseteq \mathbb{R}^d$ and $t\geq 0$, we write $Z_t(A)$ to denote the mass of $Z$ that fall inside $A$ at time $t$. 

We also define a \emph{BBM deactivated at a moving boundary}. For a Borel set $A\in\mathbb{R}^d$, denote by $\partial A$ the boundary of $A$. Consider a family of Borel sets $B=(B_t)_{t\geq 0}$. Let $Z^B=(Z_t^{B_t})_{t\geq 0}$ be the BBM deactivated at $\partial B$, which can be obtained from $Z$ as follows. For each $t\geq 0$, start with $Z_t$, and delete from it any particle whose ancestral line up to $t$ has exited $B_t$ to obtain $Z^{B_t}_t$. This means, $Z_t^{B_t}$ consists of particles at time $t$ whose ancestral lines have been confined to $B_t$ up to time $t$ (but may have left $B_s$ at an earlier time $s$), and therefore it can be viewed as a finite discrete measure in $B_t$. The terminology `deactivated' reflects on the following nature of the process: if a particle of $Z_s$ is not part of $Z_s^{B_s}$ at time $s$, this means it has been deactivated since its ancestral line has exited $B_s$ over $[0,s]$; it could reappear (or be reactivated) at a later time $u$ and hence be a part of $Z_u^{B_u}$ provided its ancestral line becomes fully contained in $B_u$ over $[0,u]$. 

This choice of a model for a BBM deactivated at a moving boundary is natural to apply in the mild obstacle problem for BBM. One can show that in almost every environment in $\mathbb{R}^d$, that is, $\mathbb{R}^d$ with a random trap field as in \eqref{eqtrapfield} is attached, certain trap-free regions of different sizes and locations exist, which may be suitably indexed by time $t$ and thus serve as suitable time-dependent balls with moving boundaries. It is the free growth of the BBM inside these trap-free regions that decides the overall growth rate of the BBM in the random environment in $\mathbb{R}^d$.

We denote by $\widehat{\Omega}$ the sample space for the BBM, and use $P_x$ and $E_x$, respectively, to denote the law and corresponding expectation of a BBM starting with a single particle at $x\in\mathbb{R}^d$. By an abuse of notation, we use $P_x$ and $E_x$ also for the BBM deactivated at a boundary. For simplicity, we set $P=P_0$. Also, we sometimes use
$$ n_t:=|Z_t^{B_t}|$$ 
to denote the mass at time $t$ of a BBM deactivated at $\partial B$.

Consider a \emph{radius} function $r:\mathbb{R}_+\to \mathbb{R}_+$ with $\lim_{t\rightarrow\infty}r(t)=\infty$, which is subdiffusively increasing, that is, $r$ is increasing and $r(t)=o(\sqrt{t})$ as $t\rightarrow\infty$. For $t>0$, let $B_t:=B(0,r(t))$, and $p_t$ be the probability of confinement to $B_t$ of a standard BM (starting from the origin) over $[0,t]$. In the second part of the current work (see Theorem~\ref{thm1}), for a suitably decreasing function $\gamma:\mathbb{R}_+\to \mathbb{R}_+$ with $\lim_{t\rightarrow\infty} \gamma(t)=0$, we find the asymptotic behavior as $t\rightarrow\infty$ of the large-deviation (LD) probability
$$P\left(\big|Z_t^{B_t}\big|<\gamma_t p_t e^{\beta t}\right) ,$$
where we have set $\gamma_t=\gamma(t)$ for convenience. It is easy to show that $E[n_t]=p_t e^{\beta t}$; therefore, since $\lim_{t\rightarrow\infty} \gamma(t)=0$, for large $t$ one could suspect that $\gamma_t p_t e^{\beta t}$ is atypically small for the mass of a BBM deactivated at $\partial B$. Indeed, Theorem~\ref{thm1} verifies that this is the case. To the best of our knowledge, it is not known whether the almost sure growth of $n_t$ agrees with its expected growth.

\subsection{History}

The study of branching diffusions among random obstacles in $\mathbb{R}^d$ goes back to Engl\"ander \cite{E2000}, who studied a BBM among \emph{hard} Poissonian obstacles in the case of a uniform field, and obtained the asymptotic probability of survival for the system as $t\rightarrow\infty$ in $d\geq 2$. In the hard obstacle model, the process is killed instantly when a particle hits the traps. Engl\"ander and den Hollander \cite{EH2003} then studied the more interesting case where the trap intensity was radially decaying as
\begin{equation} \label{radialdecay}
\frac{\text{d}\nu}{\text{d}x} \sim \frac{\ell}{|x|^{d-1}}, \quad |x|\rightarrow\infty, \quad \ell>0 , 
\end{equation}
where $\text{d}\nu/\text{d}x$ denotes the density of the mean measure of the PPP with respect to the Lebesgue measure. It is shown in \cite{EH2003} that the decay rate in \eqref{radialdecay} is interesting, because it gives rise to a phase transition at a critical intensity $\ell=\ell_{cr}>0$, at which the behavior of the system changes in terms of the optimal survival strategy. In contrast, if the decay order is larger (or smaller), there will be no such phase transition and the optimal survival strategy will simply be as in the case of the decay in \eqref{radialdecay} and $\ell>\ell_{cr}$ (or $\ell<\ell_{cr}$). In both \cite{E2000} and \cite{EH2003}, the branching rule was taken as strictly dyadic, and the main result was the exponential asymptotic decay rate of the annealed survival probability as $t\rightarrow\infty$; in addition, in \cite{EH2003}, several optimal survival strategies were proved. For a BBM with a generic branching law, denote by $p_0$ the probability that a particle gives no offspring at the end of its lifetime. In \cite{O2016}, the work in \cite{E2000} was extended to a BBM with a generic branching law, including the case where $p_0>0$. Likewise in \cite{OCE2017}, the work in \cite{EH2003} on the radially decaying trap field was extended to a BBM with a generic branching law, with the possibility of $p_0>0$. Recently in \cite{OE2019}, conditioning the BBM on the event of survival from hard Poissonian obstacles, \"Oz and Engl\"ander proved several optimal survival strategies in the annealed environment, with particular emphasis on the population size.

The annealed setting can be quite different from the quenched setting. An annealed result is obtained by averaging over all environments, that is, with the notation introduced above, the law $(\mathbb{E}\,\otimes P^\omega)(\cdot)$ is used to calculate probabilities. This means, joint strategies involving both the random process (the BBM in the current problem) and the random environment can be used to realize the event in question, whereas only strategies involving the random process which can be realized in almost every environment can be used in the quenched setting. In this sense, the system has significantly more `freedom' in the annealed setting. For instance, the quenched and annealed asymptotics are different for the classical problem of Brownian survival among Poissonian traps (see Theorems $4.5.1$ and $4.5.3$ in \cite{S1998}).       

We refer the reader to \cite{E2007} for a survey, and to \cite{E2014} for a detailed treatment on the topic of BBM among random obstacles, and to \cite{LV2012} for a related problem where a critical BBM that is killed at a small rate inside the traps (such traps are called \emph{soft} obstacles) is studied. We repeat that the mild obstacle model studied in the current work was proposed by Engl\"ander in \cite{E2008}. It is the partial aim of this work to improve the WLLN therein for the population size of the BBM to the SLLN using purely probabilistic techniques as opposed to techniques involving partial differential equations (PDEs) in \cite{E2008}, as well as extending the results to the case of no branching within the obstacles.  

The study of branching diffusions in restricted domains with absorbing boundaries goes back to Sevast'yanov \cite{S1958}, who studied the survival of such systems in bounded domains in $\mathbb{R}^d$. In \cite{K1978}, Kesten studied a BBM with negative drift in one dimension starting with a single particle at position $x>0$ in the presence of absorption at the origin. He obtained a survival criterion depending on the drift of the process, and an asymptotic result on the number of particles in a given interval. Later, various further results were obtained on the one-dimensional model with absorption at a one-sided barrier. In \cite{N1987}, considering a BBM starting with a single particle at the origin and with a strong enough negative drift so as to make extinction almost sure, Neveu studied the process $(Z_x)_{x\geq 0}$ formed by the total mass that is frozen upon exiting $((-\infty,-x), x\geq 0)$. Berestycki et al.\ followed up on Kesten's model of BBM with absorption at the origin, and in \cite{BBS2011} and \cite{BBS2013} studied, respectively, the survival probability of the BBM near the critical drift as a function of $x>0$, and the genealogy of the process. In \cite{HHK2006}, on the same model, Harris et al.\ studied the one-sided FKPP travelling-wave equation, and obtained several results on the asymptotic speed of the rightmost particle, the almost sure exponential growth rate of particles having different speeds, and the asymptotic probability of presence of the BBM in the subcritical speed area. Then, in \cite{M2013}, Maillard improved on Neveu's work in the case where the process goes extinct almost surely, and obtained precise asymptotics on the number of absorbed particles at the linear one-sided barrier. More recently in \cite{H2016}, Harris et al.\ studied a BBM with drift in a fixed-size interval, that is, a two-sided barriered version of Kesten's model, and obtained a survival criterion involving a critical width for the interval, and also the asymptotics of the near-critical survival probability. The one-dimensional model involving a BBM with drift and a fixed barrier is equivalent to the model involving a BBM with no drift and a linearly moving barrier. The second part of the current work gives a lower tail large-deviation result on the population size of a BBM in a subdiffusively expanding (time-dependent) ball in $d$ dimensions with deactivating boundary. 


\medskip

\noindent \textbf{Notation:} We use $c,c_0,c_1,\ldots$ as generic positive constants, whose values may change from line to line. If we wish to emphasize the dependence of $c$ on a parameter $p$, then we write $c(p)$. We denote by $f:A\to B$ a function $f$ from a set $A$ to a set $B$. For two functions $f,g:\mathbb{R}_+\to\mathbb{R}_+$, we write $g(t)=o(f(t))$ if $g(t)/f(t)\rightarrow 0$ as $t\rightarrow\infty$. Also, for a generic function $g:\mathbb{R}_+\to\mathbb{R}_+$, we occasionally set $g_t=g(t)$ for notational convenience. We use $\mathbb{N}$ as the set of positive integers. For $x\in\mathbb{R}^d$, we use $|x|$ to denote its Euclidean norm; also, for a generic finite set $S$, we use $|S|$ to denote its cardinality. We use $B(x,a)$ to denote the open ball of radius $a>0$ centered at $x\in\mathbb{R}^d$. For an event $A$, we use $A^c$ to denote its complement, and $\mathbbm{1}_A$ its indicator function. 

We denote by $X=(X_t)_{t\geq 0}$ a generic standard Brownian motion (BM) in $d$-dimensions, and use $\mathbf{P}_x$ and $\mathbf{E}_x$, respectively, as the law of $X$ started at position $x\in\mathbb{R}^d$, and the corresponding expectation.    

\medskip

\noindent \textbf{Outline:} The rest of the paper is organized as follows. In Section~\ref{section2}, we present our main results. In Section~\ref{newsection}, we discuss some further problems related to the model of BBM among random obstacles. Section~\ref{section3} contains several introductory results, which serve as preparation for the proofs of Theorem~\ref{thm2} and Theorem~\ref{thm1}. In Section~\ref{section5} and Section~\ref{section4}, we present, respectively, the proofs of Theorem~\ref{thm2} and Theorem~\ref{thm1}.

\section{Main Results} \label{section2}

The first main result is a quenched SLLN for the total mass of BBM among mild Poissonian obstacles in $\mathbb{R}^d$. Recall that $(\Omega,\mathbb{P})$ is the probability space for the PPP that creates the random environment, and $N_t:=|Z_t|$. We now introduce further notation. Let $\lambda_{d,r}$ be the principal Dirichlet eigenvalue of $-\frac{1}{2}\Delta$ on $B(0,r)$ in $d$ dimensions (see \cite[Section 1.10]{E2014} where $\lambda_c(-\frac{1}{2}\Delta,B(0,r))=\lambda_{d,r}$). Write $\lambda_d:=\lambda_{d,1}$, and let $\omega_d$ be the volume of the $d$-dimensional unit ball. For $d\geq 1$ and $\nu>0$, define the constant
\begin{equation} \label{eqconstant}
c(d,\nu):=\lambda_d \left(\frac{d}{\nu \omega_d}\right)^{-2/d}.
\end{equation}   

\begin{theorem}[Quenched SLLN for BBM among mild obstacles]\label{thm2}
On a set of full $\mathbb{P}$-measure,
\begin{equation} \label{eqthm2}
\underset{t\rightarrow\infty}{\lim} (\log t)^{2/d}\left(\frac{\log N_t}{t}-\beta_2\right)=-c(d,\nu) \quad P^\omega\text{-a.s.} 
\end{equation}
\end{theorem}

\begin{remark} \label{remark0}
Note that the branching rate $\beta_1$ in the trap field $K$ and the trap radius $a$ do not appear in the formula. The rough intuition for this, is as follows. The growth of mass among mild traps, to the leading order, is entirely determined by the free growth inside a `large' clearing (see Definition~\ref{def1}), which is known to exist almost surely regardless of the values of $\beta_1$ and $a$ (see Proposition D). Also, regardless of $\beta_1$ and $a$, with high probability the system is able to hit such a clearing soon enough so that the sub-BBM emanating from the particle that hits the large clearing produces sufficiently many particles inside this clearing in the remaining time. The details are presented in the proof of the lower bound of Theorem~\ref{thm2}. Needless to say, the growth inside the large clearing does not feel the effect of either $\beta_1$ or $a$. In essence, the growth of mass is determined by large trap-free regions rather than the traps, which is why the result is quite robust to the details of the trapping mechanism such as the values of $\beta_1$ and $a$.
\end{remark}

\begin{remark} \label{remark4}
It was shown in \cite{E2008} that on a set of full $\mathbb{P}$-measure,
\begin{equation} \label{eqexp}
E^\omega[N_t]=\exp\left[\beta_2 t-c(d,\nu)\frac{t}{(\log t)^{2/d}}(1+o(1))\right].
\end{equation}
Theorem~\ref{thm2} is called a SLLN for BBM among mild obstacles, because it says that with $P^\omega$-probability one,
the total mass of BBM among mild obstacles grows as its expectation as $t\rightarrow\infty$. The reason why it is called a \emph{quenched} SLLN is that it holds on a set of full $\mathbb{P}$-measure.

It would be difficult to obtain any finer result than the one in Theorem~\ref{thm2} for the following reason. The expected value formula in \eqref{eqexp} comes from a `one-particle picture' along with the many-to-one formula for spatial branching processes (see, for instance \cite[Lemma 1]{GHK2022}), where the one-particle picture gives the quenched asymptotic behavior as $t\to\infty$ of the survival probability of a single Brownian motion among soft Poissonian obstacles with killing function $W(\cdot)=(\beta_2-\beta_1)\mathbbm{1}_{B(0,a)}(\cdot)$ (except that $W$ is not summed up on overlapping balls). This quenched result is known from Sznitman's celebrated work \cite[Theorem 2.6]{S1993}, and says that the aforementioned large-time survival probability decays as 
$$ \exp\left[-c(d,\nu)\frac{t}{(\log t)^{2/d}}(1+o(1))\right]  $$
on a set of full $\mathbb{P}$-measure. To the best of our knowledge, a lower order correction to this result that accounts for the $o(1)$ term above, has not been obtained. Therefore, a finer result on even $E^\omega[N_t]$ seems far from trivial.
\end{remark}

Our second main result gives the large-time asymptotic behavior of the probability that the mass of BBM inside a subdiffusively expanding ball $B=(B_t)_{t\geq 0}$ that is deactivated at the boundary of the ball, is atypically small. A subdiffusive expansion means that the ball is expanding slower than the rate at which a typical BM moves away from the origin, which means for large $t$ it would be a `rare event' for the BM to be confined in $B_t$. For a generic standard Brownian motion $X=(X_t)_{t\geq 0}$ and a Borel set $A\subseteq\mathbb{R}^d$, define $\sigma_A=\inf\{s\geq 0:X(s)\notin A\}$ to be the first exit time of $X$ out of $A$. 

\begin{theorem}[Large-deviation for mass of BBM in an expanding ball] \label{thm1}
Let $r:\mathbb{R}_+ \to \mathbb{R}_+$ be increasing such that $r(t)\to\infty$ as $t\to\infty$ and $r(t)=o(\sqrt{t})$. Also, let $\gamma:\mathbb{R}_+ \to \mathbb{R}_+$ be defined by $\gamma(t)=e^{-\kappa r(t)}$, where $\kappa>0$ is a constant. For $t>0$, set $B_t=B(0,r(t))$, $p_t=\mathbf{P}_0(\sigma_{B_t}\geq t)$, and $n_t=|Z_t^{B_t}|$. 

\noindent If $0<\kappa\leq\sqrt{\beta/2}$, then
\begin{equation} \label{eq1}
\underset{t\rightarrow\infty}{\lim}\,\frac{1}{r(t)}\log P\left(n_t<\gamma_t p_t e^{\beta t}\right)=-\kappa,
\end{equation} 
and if $\kappa>\sqrt{\beta/2}$, then
\begin{align} 
-(\kappa \wedge \sqrt{2\beta})&\leq \underset{t\rightarrow\infty}{\liminf}\,\frac{1}{r(t)}\log P\left(n_t<\gamma_t p_t e^{\beta t}\right) \label{eq2} \\
&\leq \underset{t\rightarrow\infty}{\limsup}\,\frac{1}{r(t)}\log P\left(n_t<\gamma_t p_t e^{\beta t}\right)\leq -\sqrt{\beta/2}, \label{eq200}
\end{align} 
where we use $a\wedge b$ to denote the minimum of the numbers $a$ and $b$.
\end{theorem}

\begin{remark} \label{remark1}
The reason we call $P(n_t<\gamma_t p_t e^{\beta t})$ with $\gamma_t= e^{-\kappa r(t)}$ a \emph{large-deviation (LD)} probability is that with this choice of $\gamma_t$, both $P(n_t<\gamma_t p_t e^{\beta t})$ and $P(n_t=0)$ decay as $e^{-c r(t)}$ to the leading order for large $t$, where the values of the constant $c>0$ may differ.  

Indeed, start with
\begin{equation*} 
P(n_t<\gamma_t p_t e^{\beta t})\geq P(n_t=0).
\end{equation*}
One way to realize $\{n_t=0\}$ is to completely suppress the branching and move the initial particle out of $B_t:=B(0,r(t))$ over $[0,k r(t)]$, where $k>0$ is a constant. The probability of realizing this joint strategy is
\begin{equation}
\exp\left[-\beta k r(t)-\frac{r(t)}{2 k}(1+o(1))\right], \label{eq3}
\end{equation}  
where the second term in the exponent comes from Proposition A (see Section~\ref{section3}) along with Brownian scaling. To minimize the absolute value of the exponent in \eqref{eq3}, set $\beta k r(t)=r(t)/(2k)$, which yields $k=1/(\sqrt{2\beta})$. With this choice of $k$, we arrive at
\begin{equation} \label{eq30}
P(n_t=0)\geq \exp\left[-\sqrt{2\beta}r(t)(1+o(1))\right] .
\end{equation}

The lower bound strategies presented here for $\{n_t<\gamma_t p_t e^{\beta t}\}$ and $\{n_t=0\}$ turn out to be the same when $\kappa\geq\sqrt{2\beta}$, but different when $0<\kappa<\sqrt{2\beta}$ (see the proof of the lower bound of Theorem~\ref{thm1} in Section~\ref{section4}). Here is the heuristics behind these strategies. When $\kappa$ is small ($\kappa\leq\sqrt{2\beta}$), an obvious strategy to realize the unlikely event $\{n_t<\gamma_t p_t e^{\beta t}\}$ is to simply suppress the branching of the initial particle for long enough just so as to account for the factor of $\gamma_t$ in $\gamma_tp_te^{\beta t}$, that is, long enough so that $\gamma_t p_t e^{\beta t}$ is no longer atypical for the growth of the BBM in $B(0,r(t))$ in the remaining time interval. On the other hand, when $\kappa$ is large ($\kappa>\sqrt{2\beta}$), a less costly lower bound strategy is to simultaneously suppress the branching of the initial particle and move it out of the ball $B(0,r(t))$, over a time period of optimal length such that the joint cost of these two partial strategies are minimized. Once the initial particle is moved out of $B(0,r(t))$, the event $\{n_t=0\}$ is realized, and no further suppression of branching is needed.
\end{remark}

\begin{remark} \label{remark2}
In Theorem~\ref{thm1}, the reason why we assume $r(t)\to\infty$ as $t\to\infty$ and $r(t)=o(\sqrt{t})$ is that we are interested in a model where the typical population growth of BBM has an extra subexponentially decaying factor compared to that of an ordinary BBM. That is, we are looking for a typical population size of $f(t)e^{\beta t}$ at time $t$, where $f(t)\to 0$ subexponentially as $t\to\infty$. If $r(t)\leq M$ for all large $t$ for some $M>0$, then the extra factor decays at least exponentially fast; whereas if the expansion of $B(0,r(t))$ is diffusive or faster, the extra factor would not decay to zero. 

Also, in Theorem~\ref{thm1}, we only consider $\gamma$ with $\gamma_t=e^{-\kappa r(t)}$ for the following reason. It can be shown that if $\gamma_t\rightarrow 0$ as $t\rightarrow \infty$, then for all large $t$, $P(n_t<\gamma_t p_t e^{\beta t})\geq \delta \gamma_t$ for some $0<\delta<1$ (see the proof of the lower bound of Theorem~\ref{thm1} in Section~\ref{section4}). Hence, if $\gamma_t$ decays sufficiently slowly so as to satisfy $(\log\gamma_t)/r(t)\rightarrow 0$ as $t\rightarrow \infty$, then $\liminf_{t\rightarrow\infty}\frac{1}{r(t)}\log P\left(n_t<\gamma_t p_t e^{\beta t}\right)\geq 0$. Therefore, in view of \eqref{eq1} and \eqref{eq200} which hold in the case $\gamma_t=e^{-\kappa r(t)}$, the event $\{n_t<\gamma_t p_t e^{\beta t}\}$ would not be an LD event when $(\log\gamma_t)/r(t)\rightarrow 0$ as $t\rightarrow \infty$.    
\end{remark}

\begin{remark}
A close look at Theorem~\ref{thm1} suggests that there is a critical value of $\kappa$ at which a phase transition concerning the optimal strategy to realize the unlikely event $\{n_t<\gamma_t p_t e^{\beta t}\}$ occurs, but this critical value could be any number in $[\sqrt{\beta/2},\sqrt{2\beta}]$.

We believe that the lower bound in \eqref{eq2} is sharp, and therefore have the following conjecture.
\begin{conjecture} \label{conjecture1}
Under the assumptions of Theorem~\ref{thm1} and using the notation therein, for any $\kappa>0$,
\begin{equation}
\underset{t\rightarrow\infty}{\lim}\,\frac{1}{r(t)}\log P\left(n_t < \gamma_t p_t e^{\beta t}\right)= -(\kappa\wedge\sqrt{2\beta}). \nonumber
\end{equation}
\end{conjecture}
If the conjecture holds, this would also imply that the critical value of $\kappa$ is $\sqrt{2\beta}$, and that the presented lower bound strategies, that is, the strategy described in Remark~\ref{remark1} when $\kappa>\sqrt{2\beta}$, and the strategy described in the proof of the lower bound of Theorem~\ref{thm1} when $0<\kappa\leq\sqrt{2\beta}$, are indeed optimal. 

The reason why we believe that the lower bound in \eqref{eq2} is sharp, is as follows. In this kind of models involving branching, in order to realize large-deviation events where the population size is aytpically small with the lowest cost, the system has to prevent growth in the beginning of the relevant time interval. Once many particles are produced, the cost of controlling the growth will be much higher, and therefore such a strategy cannot be optimal. Both of the current lower bound strategies (the one in Remark~\ref{remark1} and the one in the proof of the lower bound of Theorem~\ref{thm1}) indeed prevent growth as early as possible, starting from the initial particle. Therefore, the lower bound in \eqref{eq2} is conjectured to be sharp.
\end{remark}

\section{Further problems} \label{newsection}

In this section, we discuss some further related problems. Recall that $K=K(\omega)$ denotes the trap field in the environment $\omega$. When we say a BBM has offspring law $(q_k)_{k\geq 0}$, we mean that each particle gives $k$ offspring upon branching with probability $q_k$ so that $q_k\geq 0$ for each $k\geq 0$ and $\sum_{k=0}^\infty q_k=1$.


\subsection{More general branching inside obstacles} \label{section3.1}

Consider a more general branching mechanism in $K$ such that when inside $K$ particles branch according to the offspring law $(p_k)_{k\geq 0}$ with $p_0=0$ as opposed to binary branching. We assume that $p_0=0$ since otherwise particles would be killed inside $K$ with positive probability, and we discuss the \emph{soft obstacle model} which allows killing inside $K$ as a separate problem below. Let $\mu_1=\sum_{k=1}^\infty k p_k$ be the associated mean number of offspring. We continue to assume binary branching outside $K$. Note that in this way, both the branching rate and the offspring mean depend on position as
\begin{align}
\beta(x,\omega) &=\beta_2\,\mathbbm{1}_{K^c(\omega)}(x)+\beta_1\,\mathbbm{1}_{K(\omega)}(x),  \label{branchinglaw} \\
\mu(x,\omega)   &=2\,\mathbbm{1}_{K^c(\omega)}(x)+\mu_1\,\mathbbm{1}_{K(\omega)}(x)  . \label{branchinglaw2}
\end{align} 
Intuitively, as long as $\beta_2>\beta_1(\mu_1-1)$, which means the net growth rate per particle outside $K$ is larger than the rate inside $K$, one would expect the growth of mass for large times to be governed by large clearings to the leading order similar to the case of binary branching inside $K$, and hence expect the SLLN in Theorem~\ref{thm2} to be robust against this change of offspring law inside $K$. Indeed, we have the following theorem.

\begin{theorem}\label{thm3}
For each $\omega$, suppose that the branching rate is as in \eqref{branchinglaw}, the branching is binary in $K^c(\omega)$, and $(p_k)_{k\geq 0}$ is the offspring law inside $K(\omega)$. Assume that $p_0=0$ and let $\mu_1=\sum_{k=1}^\infty k p_k<\infty$. Provided $\beta_2>\beta_1(\mu_1-1)$, on a set of full $\mathbb{P}$-measure,
\begin{equation} \label{eqthm3}
\underset{t\rightarrow\infty}{\lim} (\log t)^{2/d}\left(\frac{\log N_t}{t}-\beta_2\right)=-c(d,\nu) \quad P^\omega\text{-a.s.} 
\end{equation}
\end{theorem}

\begin{proof}
The lower bound in \eqref{eqthm3} follows from Theorem~\ref{thm2} by $\omega$-wise comparison with the case $\beta_1=0$, since $p_0=0$ by assumption here and the `worst' case of no branching, that is, $\beta_1=0$, inside $K$ is already covered by Theorem~\ref{thm2}. 

For the upper bound, we follow a similar approach as for that of Theorem~\ref{thm2}. Define $m(x,\omega)=\mu(x,\omega)-1$ and $m_1=\mu_1-1$. Applying the classical first moment formula for spatial branching processes $\omega$-wise (see, for instance \cite[Lemma 1]{GHK2022} for a more general version), we have
\begin{equation} \label{eqfirstmoment}
E^\omega[N_t] = \mathbf{E}_0 \left[\exp\left(\int_0^t \beta(X_s,\omega)m(X_s,\omega) ds \right)\right],
\end{equation}
where, as before, $X=(X_t)_{t\geq 0}$ is a Brownian motion in $d$-dimensions, and $\mathbf{E}_x$ is the corresponding expectation for a process started at $x$. Write $\beta=\beta_2-(\beta_2-\beta_1)\mathbbm{1}_{K}$ in \eqref{branchinglaw} and $m=1-(1-m_1)\mathbbm{1}_{K}$ in \eqref{branchinglaw2}, which yield $\beta m = \beta_2 - (\beta_2-\beta_1 m_1)\mathbbm{1}_K$. It follows from \eqref{eqfirstmoment} that
\begin{equation} \label{eqfirstmoment2}
E^\omega[N_t] = e^{\beta_2 t }\mathbf{E}_0 \left[\exp\left(-\int_0^t (\beta_2-\beta_1 m_1)\mathbbm{1}_{K(\omega)}(X_s) ds \right)\right] .
\end{equation}   
Note that by assumption, $\beta_2-\beta_1 m_1>0$. Then, the expectation on the right-hand side is the survival probability up to $t$ of a single Brownian motion among soft obstacles with killing function $W(x)=(\beta_2-\beta_1 m_1)\mathbbm{1}_{\bar{B}(0,a)}(x)$, except that $W$ is not summed on the overlapping balls. This, nonetheless, does not affect the asymptotic behavior of the survival probability (see \cite[Remark 4.2.2]{S1998}). Therefore, it follows from \cite[Theorem 4.5.1]{S1998} that on a set of full $\mathbb{P}$-measure, 
\begin{equation} \label{eqfirstmoment3}
E^\omega[N_t] = \exp\left[\beta_2 t-c(d,\nu)\frac{t}{(\log t)^{2/d}}(1+o(1))\right]. 
\end{equation}
By the Markov inequality, we then have for any $\varepsilon>0$,
\begin{align}
P^\omega\left((\log t)^{2/d}\left(\frac{\log N_t}{t}-\beta_2\right)+c(d,v)>\varepsilon\right) &=P^\omega\left(N_t>\exp\left[t\left(\beta_2-(c(d,\nu)-\varepsilon)(\log t)^{-2/d}\right) \right] \right)  \nonumber \\
&\leq \exp\left[-\varepsilon t(\log t)^{-2/d}+o\left(t(\log t)^{-2/d}\right)\right]. \label{eqfirstmoment4} 
\end{align}
The rest of the proof is identical to the corresponding part of the proof of the upper bound of Theorem~\ref{thm2}.
\end{proof}

In the proof above, observe that there was no need to modify Theorem~\ref{thm1}, which plays a key role in the proof of the lower bound of Theorem~\ref{thm2}. The reason is, in the latter proof, Theorem~\ref{thm1} is used to argue that there is sufficient growth of mass inside certain clearings (i.e., trap-free regions) with high probability, and changing the offspring law inside the traps has no effect on what happens inside clearings.

\subsection{More general branching outside obstacles}

Consider a more general branching mechanism in $K^c$ such that when inside $K^c$ particles branch according to the offspring law $(p^*_k)_{k\geq 0}$. Let $\mu_2=\sum_{k=0}^\infty k p^*_k$ be the associated mean. Assume that $\mu_2<\infty$ and set $m_2=\mu_2-1$. Without loss of generality, we also assume that $p^*_1=0$. Let $\mathcal{E}$ be the event of extinction for the underlying Galton-Watson process and set $q=P(\mathcal{E})$. From the elementary theory of branching processes, it is known that $q=1$ if and only if $\mu_2\leq 1$. Therefore, for meaningful results, we consider the supercritical case, that is, $\mu_2>1$. 

First, assume further that $p^*_0=0$ so that $q=0$. It follows by a similar reasoning leading to \eqref{eqfirstmoment3} that
\begin{equation} \nonumber
E^\omega[N_t] = \exp\left[\beta_2 m_2 t-c(d,\nu)\frac{t}{(\log t)^{2/d}}(1+o(1))\right]. 
\end{equation} 
Then, an application of Markov inequality similar to \eqref{eqfirstmoment4} followed by a standard Borel-Cantelli argument will lead to 
\begin{equation}
\underset{t\rightarrow\infty}{\limsup}\, (\log t)^{(2/d)}\left(\frac{\log N_t}{t}-\beta_2m_2\right)\leq -c(d,\nu) \quad P^\omega\text{-a.s.} \nonumber
\end{equation}
on a set of full $\mathbb{P}$-measure.

We believe that an SLLN similar to the one in Theorem~\ref{thm2} holds with $\beta_2$ replaced by $\beta_2 m_2$ in \eqref{eqthm2}, but to prove the lower bound, one needs to extend Theorem~\ref{thm1} to a BBM with a general offspring law $(p^*_k)_{k\geq 0}$. The upper bound of this extension turns out to be difficult, and the main reason is as follows. (A close look at the proof of Theorem~\ref{thm2} shows that we only use the upper bound in Theorem~\ref{thm1} to prove Theorem~\ref{thm2}.) In the proof of the upper bound of Theorem~\ref{thm1}, as an essential part of the second moment argument that is used to bound $P(n_t<\gamma_t p_t e^{\beta t})$ (for details, see the proof of the upper bound of Theorem~\ref{thm1}), one needs estimates on the distribution of the most recent common ancestor of two particles of BBM randomly chosen among the particles alive at time $t$, and this distribution is known in the case of binary branching. To the best of our knowledge, under a general offspring law, this distribution or any kind of useful estimate such as \eqref{eqdensityfunction} thereof is not known, and finding it would be a problem of independent interest. If one overcomes this problem, and hence is able to find a positive upper bound for $-\limsup_{t\to\infty}\,\frac{1}{r(t)}\log P\left(n_t<\gamma_t p_t e^{\beta t}\right)$, then the proof of the lower bound of Theorem~\ref{thm2} can be carried out in the same way as in Section~\ref{lowerbound} with the replacement of $\beta_2$ by $\beta_2 m_2$ throughout the proof.

Now suppose that $p^*_0>0$. In this case, there is positive probability of extinction for the underlying Galton-Watson process, that is, $q=P(\mathcal{E})>0$, and the process is conditioned on non-extinction for meaningful results on the growth of mass of BBM for large times. A detailed treatment of a BBM conditioned on non-extinction is given in \cite{OCE2017} (see Lemma 4 and Proposition 2 therein). In particular, conditioned on $\mathcal{E}^c$, the BBM has the two-type decomposition:
$$  (Z_t)_{t\geq 0} = (Z^1_t,Z^2_t)_{t\geq 0}, $$
where $Z^1$ is the process consisting of particles with infinite lines of descent, called the \emph{skeleton}, and $Z^2$ is the one consisting of particles with finite lines of descent. Conditional on $\mathcal{E}^c$, the process $Z^1 = (Z^1_t)_{t\geq 0}$ is a BBM of its own, with the same net growth per particle $\beta_2 m_2$ as the original process $Z$. Therefore, if an SLLN for $N_t$ similar to Theorem~\ref{thm2} is proved for the case of a general offspring distribution with $p^*_0=0$, then the same result would hold for the total mass $|Z^1|$ of the skeleton in the case $p^*_0>0$ conditional on $\mathcal{E}^c$. The analysis for the process $Z^2$ is not so simple since it is not a BBM, and is formed by a collection of independent BBMs initiated at random times along the skeletal lines (see \cite[Section 5]{OE2019} for details). 




\subsection{Soft killing inside obstacles} 

On top of complete suppression of branching, one may consider soft killing inside the obstacles. More generally, we may describe the soft obstacle model for BBM as follows. Consider a positive, bounded, measurable and compactly supported \emph{killing function} $W:\mathbb{R}^d\to (0,\infty)$, and for $\omega=\sum_i \delta_{x_i}\in\Omega$, $x\in\mathbb{R}^d$, define the potential
\begin{equation} 
V(x,\omega)=\sum_i W(x-x_i). \label{eqpotential}
\end{equation}
Then, the Poissonian trap field $K=K(\omega)$ in $\mathbb{R}^d$ and the soft obstacle model for BBM are formed as follows: 
\begin{equation} \label{eqtraprule}
 x\in K(\omega) \:\:\Leftrightarrow\:\:   V(x,\omega)>0  ,
\end{equation}
particles branch at the normal rate $\beta$ when outside $K$, whereas inside $K$ they are killed at rate $V=V(x,\omega)$ and their branching is completely suppressed. Note that the special case of constant killing rate inside spherical traps defined in \eqref{eqtrapfield} corresponds to taking $W=\alpha\mathbbm{1}_{\bar{B}(0,a)}$ with some constant $\alpha>0$ except that $W$ is not summed up on overlapping balls. A formal treatment of BBM killed at rate $V=V(x,\omega)$ in $\mathbb{R}^d$ is given in \cite{LV2012}.

Compared to the mild obstacle model, the main extra challenge comes from the fact that there is positive probability for the entire process to be killed in a finite time due to possible killing of particles. Therefore, to obtain meaningful results, the process is conditioned on the event of ultimate survival. Recall that $N_t=|Z_t|$ denotes the mass of BBM at time $t$. Let 
\begin{equation} \nonumber 
S_t=\{N_t\geq 1\},   \quad \quad S = \bigcap_{t\geq 0} S_t
\end{equation}
be, respectively, the event of survival up to time $t$, and the event of ultimate survival. As before, we use $P^\omega$ to denote the conditional law of the BBM in the random environment $\omega$. By continuity of measure from above, one easily sees that $\lim_{t\to\infty} P^\omega(S_t) = P^\omega(S)$. Moreover, it is not hard to prove the following result.

\begin{proposition}[Survival probability, soft obstacles] \label{newprop}
In the soft obstacle model described above, on a set of full $\mathbb{P}$-measure, $0<P^\omega(S)<1$.  
\end{proposition}

Note that any conditioning on $S$ or on $S_t$ for some $t$ will change the law of the BBM. In particular, the ancestral lines will no longer have the law of standard Brownian motions.   

In the mild obstacle problem, we have seen that the SLLN in Theorem~\ref{thm2} does not depend on the reduced rate $\beta_1$ inside the obstacles or the trap radius $a$ (see Remark~\ref{remark0}), and as the corresponding proof of the lower bound shows, this is because the growth is mainly due to the free growth inside a large clearing which the BBM is able to hit soon enough. Even when there is killing inside the obstacles, such a large clearing will continue to exist almost surely, and intuitively the BBM should still be able to hit this large clearing soon enough with high probability. Therefore, we have the following conjecture. Define the law $\widehat{P}^\omega$ as $\widehat{P}^\omega(\:\cdot\:)=P^\omega(\:\cdot\:\mid S)$.  

\begin{conjecture} \label{conjecture2}
In the soft obstacle model described above, on a set of full $\mathbb{P}$-measure,
\begin{equation} \nonumber
\underset{t\rightarrow\infty}{\lim} (\log t)^{2/d}\left(\frac{\log N_t}{t}-\beta\right)=-c(d,\nu) \quad \widehat{P}^\omega\text{-a.s.} 
\end{equation}
\end{conjecture}

We think that a proof for the conjecture above can be given by following a similar method as in the proof of Theorem~\ref{thm2}. Nonetheless, there will be some serious additional challenges compared to the mild obstacle case. We now give an outline of the main similarities and differences, comparing the two models. We only consider key difficulties, and note that there are various other minor hurdles that need to be overcome and are not present in the case of mild obstacles.

The expected mass at time $t$ can be obtained similar to \eqref{eqfirstmoment3}. Observe that soft killing under the potential $V$ together with complete suppression of branching inside the obstacles is tantamount to the offspring law $(p_k)_{k\geq 0}$ with $p_0=1$ and rate $\beta_1=V(x,\omega)$ in \eqref{branchinglaw}, which yields
\begin{align}
\beta(x,\omega) &=\beta_2\,\mathbbm{1}_{K^c(\omega)}(x)+V(x,\omega)\,\mathbbm{1}_{K(\omega)}(x), \nonumber \\ 
\mu(x,\omega)   &=2\,\mathbbm{1}_{K^c(\omega)}(x) \nonumber . 
\end{align} 
Note that $p_0=1$ implies $\mu_1=0$. By the construction in \eqref{eqtraprule}, $V=V\mathbbm{1}_K$. Then, $\beta(\mu-1)=\beta m=\beta_2-(\beta_2+V)\mathbbm{1}_K$. Apply the many-to-one formula in \eqref{eqfirstmoment} to obtain
$$ E^\omega[N_t] = e^{\beta_2 t }\mathbf{E}_0 \left[\exp\left(-\int_0^t (\beta_2\mathbbm{1}_{K(\omega)}(X_s)+V(X_s,\omega))ds\right) \right] . $$
Using \cite[Theorem 4.5.1]{S1998}, again, we obtain 
$$  E^\omega[N_t] = \exp\left[\beta_2 t-c(d,\nu)\frac{t}{(\log t)^{2/d}}(1+o(1))\right], $$
which holds on a set of full $\mathbb{P}$-measure. Then, to prove the upper bound of Conjecture~\ref{conjecture2}, we may proceed in the same way as in the  mild obstacle case. In contrast, the proof of the lower bound seems far more than a simple modification of that of the mild obstacle case. 

Firstly, since the soft killing rule applies only inside the trap field $K$, and the killing function $V$ is compactly supported by assumption (for comparison with spherical traps, take $a=\sup_{x\in K_0}\{|x|:V(x)>0\}$ where $K_0$ is the compact on which $V$ is supported), large clearings that were used in the proof of Theorem~\ref{thm2} will continue to exist in almost every environment under soft killing as well. This means, Theorem~\ref{thm1} will continue to be useful in its current form just as in the case of general branching inside obstacles, which was discussed in Section~\ref{section3.1}. 

We now discuss key additional challenges. The reader is referred to Section~\ref{lowerbound} in order to have a better understanding of what follows. As noted in Section~\ref{intro}, the most challenging part of the proof of the lower bound of Theorem~\ref{thm2} is Part 1, where we argue that in almost every environment exponentially many particles are produced with high probability. A close look at the proof shows, Lemma~\ref{lemma2} is the key tool in Part 1 and is based on careful choices of the time-scale $h(t)$ and the space scale $\rho(t)$, and the estimate $\mathbf{P}^\omega(E_{1,t}^c)\geq (\kappa_d/2) r^{d-1}(t)/\rho^{d-1}(t)$, which follows from elementary geometry. Under soft killing, this estimate no longer holds. In the language of Lemma~\ref{lemma2}, $P^\omega(E_{1,t}^c)$ is the probability that a Brownian particle hits a `good point' over the interval $[0,h(t)]$, and under soft killing the particle has to avoid being killed by $V$ and hit the $r(t)$-clearing at the same time. That is, the particle has to `hit and survive' over $[0,h(t)]$, which is more costly than just hitting the $r(t)$-clearing as in the case of mild obstacles. To avoid $V$, one must have control over the path of the particle over $[0,h(t)]$. As a result, the upper bound on $P^\omega(E_t, S_t)$ (one needs to estimate $P^\omega(E_t, S_t)$ instead of $P^\omega(E_t)$) will be significantly larger than the corresponding bound in the mild obstacle case, which in turn will affect the Borel-Cantelli argument in Part 4. Overall, by suitably adjusting the scales $h(t)$ and $\rho(t)$, it could still be possible to show that $P^\omega(E_t, S_t)\to 0$ as $t\to\infty$, but it is not clear whether the rate of convergence to zero is fast enough for the Borel-Cantelli argument to work in Part 4. Moreover, $h(t)$ and $\rho(t)$ will have to be chosen small (a closer look shows that powerlike growth will not work, one has to at best settle for logarithmic growth), which, as explained below, disturbs the preparation of the almost-sure environment where we wish our quenched results to hold.   


We now turn our attention to Lemma~\ref{lemma1}. Observe that after proving Lemma~\ref{lemma1}, we set $\ell=\rho(t)$ therein to prepare our almost-sure environment (see \eqref{eqenviron}). The key consideration is that roughly we split the cube $[-t,t]^d$ into $\sim [t/\rho(t)]^d$ smaller balls and would like to have a clearing of radius $\sim r(t)$ inside each small ball of radius $\rho(t)$. In view of $\ell=\rho(t)$, a smaller $\rho(t)$ would require stating Lemma~\ref{lemma1} with a larger number of cubes $C_{j,\ell}$. In particular, if $\rho(t)$ has logarithmic growth, we would need Lemma~\ref{lemma1} to be stated with exponentially many cubes $C_{j,\ell}$ as opposed to polynomially many, and this would certainly affect the Borel-Cantelli argument in Lemma~\ref{lemma1}. Hence, there will be extra challenges in preparing the a.s.\ environment as well due to the soft killing inside obstacles.

We finally briefly discuss how Parts 2-4 of the proof will be affected by soft killing. Part 2 of the proof concerns the hitting of the large clearing, denoted by $B(y_0,R(t))$ in the proof, over the interval $[\ell(t),m(t)]$ by the sub-BBM emanating from one of the many particles that are present at time $\ell(t)$. In the presence of the killing potential $V$, not only a progeny of such a particle should hit the large clearing by time $m(t)$, but it also has to avoid being killed by $V$ over $[\ell(t),m(t)]$. Therefore, one has to show that the probability for the BBM to hit $B(y_0,R(t))$ is large enough even in the presence of $V$, and to do this, one will presumably need to change the time scales $\ell(t)$ and $m(t)$. Part 3 of the proof will not be disturbed since $V$ does not affect what happens in the clearings, and since Theorem~\ref{thm1} continues to hold. On the other hand, Part 4 will drastically be affected by soft killing since throughout the proof, the rates of decay to zero for most of the relevant probabilities will be much slower compared to the case of mild obstacles, which will disturb the Borel-Cantelli argument in Part 4. Moreover, the total mass $N_t$ of BBM is no longer a.s.\ increasing due to possible killing of particles, and this monotonicity of $N_t$ was used in Part 4.

We stress that only key additional difficulties resulting from soft killing were discussed here. In fact, a follow-up paper to the current work is planned on the problem of BBM among soft obstacles.

\section{Preparations} \label{section3}

In this section, we present introductory results that serve as preparations for the proofs of the main theorems. The first two results are standard in the theory of Brownian motion. Proposition A is on the large-time asymptotic probability of atypically large Brownian displacements. For a proof, see for example \cite[Lemma 5]{OCE2017}. 

\begin{propa}[Linear Brownian displacements]\label{prop1}
For $k>0$,
\begin{equation}  \mathbf{P}_0\left(\underset{0\leq s\leq t}{\sup}|X(s)|>k t\right)=\exp\left[-\frac{k^2 t}{2}(1+o(1))\right].  \nonumber 
\end{equation}
\end{propa}

The following is a standard result on the large-time Brownian confinement in balls, and for instance can be deduced from \cite[Prop.\ 1.6]{E2014}, along with the scaling $\lambda_{d,r}=\lambda_d/r^2$. Recall that $\sigma_A=\inf\{s\geq 0:X(s)\notin A\}$ denotes the first exit time of $X$ out of $A$.

\begin{propb}[Brownian confinement in small balls] \label{prop2}
For $t>0$, let $B_t=B(0,r(t))$, where $r:\mathbb{R}_+ \to \mathbb{R}_+$ is such that $r(t)\to\infty$ as $t\to\infty$ and $r(t)=o(\sqrt{t})$. Then, as $t\to\infty$,
\begin{equation} \mathbf{P}_0\left(\sigma_{B_t}\geq t\right)=\exp\left[-\frac{\lambda_d t}{r^2(t)}(1+o(1))\right]. \nonumber
\end{equation}
\end{propb}

The following result is well-known in the theory of branching processes. For a proof, see for example \cite[Section 8.11]{KT1975}.

\begin{propc}[Distribution of mass in branching systems] \label{prop3}
For a strictly dyadic continuous-time branching process $N=(N_t)_{t\geq 0}$ with constant branching rate $\beta>0$, the probability distribution at time $t$ is given by 
\begin{equation} P(N_t=k)=e^{-\beta t}(1-e^{-\beta t})^{k-1},\quad k\geq 1, \nonumber
\end{equation}
from which it follows that
\begin{equation} P(N_t>k)=(1-e^{-\beta t})^k  \label{eq02}.
\end{equation}
\end{propc}

We now focus on the model of Poissonian traps in $\mathbb{R}^d$. Recall that a random environment in $\mathbb{R}^d$ is created via a PPP, called $\Pi$, with
$$K:=\bigcup_{x_i\in\,\text{supp}(\Pi)}\bar{B}(x_i,a)$$	
being the trap field attached to $\mathbb{R}^d$. 

\begin{definition} \label{def1}
A \emph{clearing} in the random environment $\omega$ is a trap-free region in $\mathbb{R}^d$, that is, $A\subseteq\mathbb{R}^d$ is a clearing if $A\subseteq K^c$. By a \emph{clearing of radius $r$}, we mean a ball of radius $r$ which is a clearing.
\end{definition}

The following result is Lemma 4.5.2 in \cite{S1998}. 

\begin{propd}
Let
\begin{equation}
R_0=R_0(d,\nu):=\left(\frac{d}{\nu \omega_d}\right)^{1/d}=\sqrt{\frac{\lambda_d}{c(d,\nu)}}. \label{eqr0}
\end{equation}
Then, on a set of full $\mathbb{P}$-measure, there exists $\ell_0=\ell_0(\omega)>0$ such that for each $\ell\geq\ell_0$ the cube $(-\ell,\ell)^d$ contains a clearing of radius 
\begin{equation}
R_\ell:=R_0(\log \ell)^{1/d}-(\log \log \ell)^2 ,\:\:\ell>1. \label{eqrell}
\end{equation}
\end{propd}

We now prove a somewhat stronger version of Proposition D, which will be needed in the proof of the lower bound of Theorem~\ref{thm2} (see Section~\ref{section5}). For a Borel set $B$ and $x\in\mathbb{R}^d$, we define their sum in the sense of sum of sets as $x+B:=\{x+y:y\in B\}$. 
\begin{lemma}[Almost sure clearings]\label{lemma1}
Let $n\in\mathbb{N}$ and $a\in\mathbb{R}_+$ be fixed, and for $\ell>0$ let $x_1,\ldots,x_{\left\lceil\ell^n\right\rceil}$ be any set of $\left\lceil\ell^n\right\rceil$ points in $\mathbb{R}^d$. Define the cubes $C_{j,\ell}=x_j+(-\ell,\ell)^d$, $1\leq j\leq \left\lceil\ell^n\right\rceil$. Then, on a set of full $\mathbb{P}$-measure, there exists $\ell_0=\ell_0(\omega)>0$ such that for each $\ell\geq \ell_0$, each of $C_{1,\ell},C_{2,\ell},\ldots,C_{\left\lceil\ell^n\right\rceil,\ell}$ contains a clearing of radius $R_\ell+a$, where $R_\ell$ is as in \eqref{eqrell}.     
\end{lemma}

\begin{proof}
Let $x_1,x_2,\ldots$ be a sequence of points in $\mathbb{R}^d$, and $C_{j,\ell}:=x_j+(-\ell,\ell)^d$ for $j=1,2,\ldots$ For $k\geq 0$, let $A_{\ell,k}$ be the event that there is a clearing of radius $R_\ell+k$ in each $C_{1,\ell},C_{2,\ell},\ldots,C_{\left\lceil(2\ell)^n\right\rceil,\ell}$. Also, for $k\geq 0$, define
$$E_{\ell,k}=\{(-\ell,\ell)^d\:\:\text{contains a clearing of radius $R_\ell+k$}\}.$$
Due to the homogeneity of the PPP, it is clear that for all $x\in\mathbb{R}^d$ and $k>0$, 
$$\mathbb{P}\left(x+(-\ell,\ell)^d\:\:\text{contains a clearing of radius $R_\ell+k$}\right)=\mathbb{P}(E_{\ell,k}).$$
Then, the union bound gives
\begin{equation}
\mathbb{P}(A_{\ell,k}^c)\leq \left\lceil(2\ell)^n\right\rceil \mathbb{P}(E_{\ell,k}^c). \label{eq1lemma3}
\end{equation}
We now estimate $\mathbb{P}(E_{\ell,k}^c)$. Partition $(-\ell,\ell)^d$ into smaller cubes of side length $2(R_\ell+k)$. Then, a ball of radius $R_\ell+k$ can be inscribed in each smaller cube, and we can bound $\mathbb{P}(E_{\ell,k}^c)$ from above as 
\begin{equation}
\mathbb{P}(E_{\ell,k}^c)\leq \left[1-e^{-\nu\omega_d (R_\ell+k)^d}\right]^{\lfloor\ell/(R_\ell+k)\rfloor^d} 
\leq \exp\left[-\left\lfloor\frac{\ell}{R_\ell+k}\right\rfloor^d e^{-\nu\omega_d (R_\ell+k)^d}\right], \label{eq2lemma3}
\end{equation}
where we have used the estimate $1+x\leq e^x$.
Let 
$$\alpha_\ell:=\left\lfloor\frac{\ell}{R_\ell+k}\right\rfloor^d e^{-\nu\omega_d (R_\ell+k)^d}.$$
Then, using \eqref{eqrell}, and that $\log\lfloor\ell/(R_\ell+k) \rfloor\geq \log\frac{\ell}{2(R_\ell+k)}$, we obtain
\begin{align}
\log \alpha_\ell&\geq d\log \ell-d\log[2(R_0(\log \ell)^{1/d}-(\log \log \ell)^2+k)]-\nu\omega_d[R_0(\log \ell)^{1/d}-(\log \log \ell)^2+k]^d \nonumber \\
&=d\log \ell-d\log[2(R_0(\log \ell)^{1/d}-(\log \log \ell)^2+k)]-\nu\omega_d R_0^d \log\ell\left[1-\frac{(\log\log\ell)^2-k}{R_0(\log\ell)^{1/d}}\right]^d \nonumber \\
&\geq d\log \ell-d\log 2-d\log[R_0(\log \ell)^{1/d}-(\log \log \ell)^2+k] \nonumber \\
&\quad -d\log\ell+\frac{d^2}{R_0}(\log\ell)^{1-1/d}[(\log\log\ell)^2-k]-\frac{d^3}{2 R_0^2}(\log\ell)^{1-2/d}[(\log\log\ell)^2-k]^2 \nonumber \\
&\geq \frac{1}{2 R_0}(\log\log\ell)^2 \label{eq3lemma3}
\end{align}
for all large $\ell$, where we have used in the first inequality that $R_0^d=d/(\nu \omega_d)$, and that $(1-x)^n\leq 1-xn+(xn)^2/2$ for $x\in[0,1]$. It follows from \eqref{eq2lemma3} and \eqref{eq3lemma3} that for a given $k>0$, for all large $\ell$,
\begin{equation*}
\mathbb{P}(E_{\ell,k}^c)\leq e^{-\alpha_\ell}\leq e^{-\exp[(\log\log\ell)^2/(2 R_0)]}.
\end{equation*}
Take $\ell=2^m$ with $m\in\mathbb{N}$. Then, since for all sufficiently large $m$, $e^{-\alpha_{(2^m)}}\leq \exp\left(-m^{\frac{\log m}{2R_0}}\right)\leq e^{-m^2}$; this, along with \eqref{eq1lemma3} and \eqref{eq3lemma3} implies    
\begin{equation*}
\sum_{m=1}^\infty \mathbb{P}\left(A_{2^m,k}^c\right) \leq c(m_0)+\sum_{m=m_0}^\infty \left\lceil(2^n)^{m+1}\right\rceil e^{-m^2}<\infty, 
\end{equation*} 
where $c(m_0)$ is a constant that depends on $m_0$. Applying Borel-Cantelli lemma to the cubes $(-2^m,2^m)^d$, we conclude that with $\mathbb{P}$-probability one, only finitely many $A_{2^m,k}^c$ occur. That is, $\mathbb{P}(\Omega_0)=1$, where
\begin{equation} \label{eqlemma1}
\Omega_0=\{\omega:\exists m_0=m_0(\omega)\:\:\forall m\geq m_0,\:\:\text{each}\:\:C_{1,2^m},\ldots,C_{\left\lceil(2^{m+1})^n\right\rceil,2^m}\:\:\text{has a clearing of radius $R_{(2^m)}+k$} \} .
\end{equation}
Let $\omega_0\in\Omega_0$, and $m_0=m_0(\omega_0)$ be the `sufficiently large $m$' from \eqref{eqlemma1}. If we choose $k\geq a$, then to complete the proof, it suffices to show that in the environment $\omega_0$ for each $m\geq m_0$ and $2^m\leq \ell\leq 2^{m+1}$, each $C_{1,\ell},\ldots,C_{\left\lceil\ell^n\right\rceil,\ell}$ contains a clearing of radius $R_\ell+a$. Let $\ell\geq 2^{m_0}$ so that $2^m\leq \ell\leq 2^{m+1}$ for some $m\geq m_0$. Fix this integer $m$. Observe that
\begin{equation} 
R_{(2^{m+1})}-R_{(2^m)}\leq R_0(\log 2)^{1/d}\left[(m+1)^{1/d}-m^{1/d}\right]\leq R_0 \log 2 . \nonumber
\end{equation}
Choose $k=R_0\log 2+a$ (so far the choice of $k>0$ was arbitrary). Then, since $R_\ell$ is increasing in $\ell$ for large $\ell$, we have
\begin{equation} \label{tavsancik2}
R_\ell+a \leq R_{(2^{m+1})}+a \leq R_{(2^m)}+R_0\log 2+a = R_{(2^m)}+k .
\end{equation}
Furthermore, 
\begin{equation} \label{tavsancik3}
\left\lceil\ell^n\right\rceil \leq \left\lceil(2^{m+1})^n\right\rceil .
\end{equation}
Then, setting $\ell_0=2^{m_0}$, \eqref{eqlemma1}, \eqref{tavsancik2} and \eqref{tavsancik3} imply that for $\ell\geq \ell_0$, each of $C_{1,\ell},\ldots,C_{\left\lceil\ell^n\right\rceil,\ell}$ contains a clearing of radius $R_\ell+a$. This completes the proof since the choice of $\omega_0\in\Omega_0$ was arbitrary and $\mathbb{P}(\Omega_0)=1$.
\end{proof}

\section{Proof of Theorem~\ref{thm2}} \label{section5}

\subsection{Proof of the upper bound}

The following upper bound was proved in \cite[Section 6.1]{E2008} via a first moment argument, using \eqref{eqexp} and the Markov inequality. On a set of full $\mathbb{P}$-measure, say $\Omega_0$, for any $\varepsilon>0$,
\begin{equation} \label{equpperbound}
P^\omega\left((\log t)^{2/d}\left(\frac{\log N_t}{t}-\beta_2\right)+c(d,\nu)>\varepsilon\right)\leq \exp\left[-\varepsilon t(\log t)^{-2/d}+o\left(t(\log t)^{-2/d}\right)\right].
\end{equation}
To pass from \eqref{equpperbound} to the upper bound of the corresponding SLLN, we use a standard Borel-Cantelli argument. Recall that $\widehat{\Omega}$ is the sample space for the BBM. For $t>0$, define
$$Y_t:=(\log t)^{2/d}\left(\frac{\log N_t}{t}-\beta_2\right),$$
and let
$$\widehat{\Omega}_0:=\{\varpi\in\widehat{\Omega}:\forall\,\varepsilon>0\:\:\exists\,t_0=t_0(\varpi)\:\text{such that}\:\forall\,t\geq t_0,\:Y_t\leq -c(d,\nu)+\varepsilon \}.$$
Let $\omega\in\Omega_0$. We will show that $P^\omega(\widehat{\Omega}_0)=1$. For $n\in\mathbb{N}$, define
$$A_n:=\{Y_n>-c(d,\nu)+\varepsilon \}.$$
By \eqref{equpperbound}, there exists $n_0\in\mathbb{N}$ such that for all $n\geq n_0$, $P^\omega(A_n)\leq e^{-c(\varepsilon) n(\log n)^{-2/d}}$. Then,
\begin{equation} 
\sum_{n=1}^\infty P^\omega(A_n)=c+\sum_{n=n_0}^\infty P^\omega(A_n)\leq c+\sum_{n=n_0}^\infty e^{-c(\varepsilon) n(\log n)^{-2/d}}<\infty.  \nonumber
\end{equation}
By the Borel-Cantelli lemma, it follows that $P^\omega(A_n\:\text{occurs i.o.})=0$, where i.o.\ stands for \emph{infinitely often}. Choosing $\varepsilon=1/k$, this implies that for each $k\geq 1$, we have
$$P^\omega(\widehat{\Omega}_k)=1,\quad \widehat{\Omega}_k:=\{\varpi\in\widehat{\Omega}:\exists\,n_0=n_0(\varpi)\:\text{such that}\:\forall\,n\geq n_0,\:Y_n\leq -c(d,\nu)+1/k \}. $$
Since $P^\omega(\widehat{\Omega}_k)=1$ for each $k\geq 1$, we have $P^\omega(\widehat{\Omega}_0)=P^\omega(\cap_{k\geq 1}\widehat{\Omega}_k)=1$. Hence, on a set of full $\mathbb{P}$-measure,
\begin{equation} 
\underset{t\rightarrow\infty}{\limsup}\, (\log t)^{(2/d)}\left(\frac{\log N_t}{t}-\beta_2\right)\leq -c(d,\nu) \quad P^\omega\text{-a.s.} \nonumber
\end{equation}

\subsection{Proof of the lower bound} \label{lowerbound}

Let $\varepsilon>0$. We will find an upper bound for
\begin{equation} \label{goal}
P^\omega\left((\log t)^{2/d}\left(\frac{\log N_t}{t}-\beta_2\right)+c(d,v)<-\varepsilon\right)=P^\omega\left(N_t<\exp\left[t\left(\beta_2-\frac{c(d,v)+\varepsilon}{(\log t)^{2/d}}\right)\right]\right)  
\end{equation}
that is valid for large $t$ on a set of full $\mathbb{P}$-measure, and then use this upper bound along with the Borel-Cantelli lemma to pass to the corresponding SLLN.

The proof is split into four parts for better readability. The first three parts are based on a bootstrap argument, where in part one, we find an upper bound on $P(N_t<e^{\delta t})$ for $0<\delta<\beta_2$, and then use this upper bound in parts two and three to find a similar upper bound on \eqref{goal}. We follow a similar proof strategy as in \cite{E2008}. However, we are required to significantly improve the first and third parts of the corresponding proof in \cite{E2008} in order to extend the WLLN therein to SLLN, where the extra work is due to finding rates of decay to zero for the relevant probabilities as $t\rightarrow\infty$ as opposed to merely showing that they tend to zero.

In the first part of the proof, we use probabilistic arguments alone, including Theorem~\ref{thm1}, in contrast to the partial differential equations (PDE) approach used in \cite{E2008}. The main challenge is due to the possibility of $\beta_1=0$ (no branching inside the traps), which makes it difficult to show that even in the presence of mild obstacles the system produces exponentially many particles with `high' probability. It is possible to include the case $\beta_1=0$ here thanks to the probabilistic approach that we follow in the first part of the proof: we show that with `high' probability the BBM first finds a `good' point (see Lemma~\ref{lemma2}), which is the center of a large-enough clearing, and then produces exponentially many particles within this clearing. We emphasize that the case $\beta_1=0$ was not covered in \cite{E2008}, and the PDE approach used therein exploits the condition $\beta_1>0$. The second part of the proof is similar to that in \cite{E2008}; here, with minor further work, we find the rate of convergence to zero of the probability of the relevant unlikely event. The third part of the proof is an application of Theorem~\ref{thm1}, where we argue that with `high' probability sufficiently many particles are produced in a certain expanding clearing in $\mathbb{R}^d$, which exists in almost every environment. The fourth part of the proof uses a Borel-Cantelli argument along with the upper bound on \eqref{goal} from part three to obtain the lower bound of the SLLN in \eqref{eqthm2}.

\medskip

\noindent \textbf{\underline{Part 1}: Upper bound on exponentially few total mass}

In the first part of the proof, we will find an upper bound for
$$P^\omega(N_t<e^{\delta t})\quad \text{with} \:\: 0<\delta<\beta_2,$$
that is valid for large $t$ on a set of full $\mathbb{P}$-measure. The argument will be based on the following lemma of independent interest, which is on the hitting probability of a standard BM to clearings of a certain size. Recall that $X=(X_t)_{t\geq 0}$ denotes a standard BM in $d$ dimensions, and $\mathbf{P}_x$ is the law of $X$ started at position $x\in\mathbb{R}^d$. Also, recall the definition of $R_0$ from \eqref{eqr0}. 

\begin{lemma}[Hitting probability of BM to large clearings]\label{lemma2}
Let $r:\mathbb{R}_+\to\mathbb{R}_+$ be such that 
\begin{equation} \label{part1eqradius}
r(t)=\frac{R_0}{3}\left(\frac{1}{6d}\right)^{1/d}(\log t)^{1/d}, \quad t>1 . 
\end{equation}
For $\omega\in\Omega$ and $t>0$, define  
\begin{equation}
\Phi_t^\omega=\{x\in\mathbb{R}^d:B(x,r(t))\subseteq K^c(\omega)\} .  \nonumber
\end{equation}
Let $\mathbf{P}^\omega_x$ be the conditional law of $X$ started at position $x\in\mathbb{R}^d$ in the random environment $\omega$. Then, there exists $\Omega_1\subseteq\Omega$ with $\mathbb{P}(\Omega_1)=1$ such that for every $\omega\in\Omega_1$, there exists $t_0=t_0(\omega)$ such that for all $t\geq t_0$,
\begin{equation} 
\mathbf{P}^\omega_0\bigg(\bigg(\bigcup_{0\leq s\leq t} \{X(s)\} \bigg)\cap \Phi_t^\omega=\emptyset\bigg)\leq e^{-t^{1/3}} .  \nonumber
\end{equation}
\end{lemma}

\begin{proof}
Introduce a time scale $h(t)$, and two different space scales $\rho(t)$ and $r(t)$, as follows. Let $h,\rho,r:\mathbb{R}_+\to\mathbb{R}_+$ satisfy:
\begin{enumerate}
	\item[(a)] $\lim_{t\rightarrow\infty}h(t)=\infty$ and $h(t)=o(t)$,
	\item[(b)] $\rho(t)=o(\sqrt{h(t)})$ and $\rho(t)=t^{d/(d+m)}$ for some $m\in\mathbb{N}$,
	\item[(c)] $r(t)=\frac{R_0}{3}[\log \rho(t)]^{1/d}$ for $t>1$.
\end{enumerate}
Later, we will choose $h(t)$ and $\rho(t)$, and hence $r(t)$, in a precise way so that $r(t)$ will be as in \eqref{part1eqradius}. 

First, we establish a suitable almost-sure environment in $\mathbb{R}^d$ with sufficient concentration of `large' clearings in $C(0,t):=[-t,t]^d$. Consider the simple cubic packing of $C(0,t)$ with balls of radius $\rho(t)/(2\sqrt{d})$. Then, at most
$$ n_t:=\left\lceil \frac{t}{\rho(t)/(2\sqrt{d})}  \right\rceil^d $$
balls are needed to completely pack $C(0,t)$, say with centers $\left(z_j:1\leq j\leq n_t\right)$. For each $j$, let $B^{j,t}=B(z_j,\rho(t)/(2\sqrt{d}))$. Now consider generically a simple cubic packing of $\mathbb{R}^d$ by balls $(\mathcal{B}_j:j\in\mathbb{N})$ of radius $R>0$, and let $x\in\mathbb{R}^d$ be any point. It is easy to see that $\min_j \max_{z\in \mathcal{B}_j}|x-z|<(\sqrt{d}/2)4R$, where $\sqrt{d}/2$ is the distance between the center and any vertex of the $d$-dimensional unit cube, i.e., $C(0,1/2)$. Then, since the packing ball radius is $\rho(t)/(2\sqrt{d})$ in our case, it follows that
\begin{equation}  \label{eq11}
\forall\,x\,\in C(0,t), \quad
\underset{1\leq j\leq n_t}{\min}\:\underset{z\in B^{j,t}}{\max}\:|x-z|<\rho(t). 
\end{equation}
Observe from the definition of $n_t$ that 
\begin{equation}  \label{eq120}
n_t\leq (3\sqrt{d})^d\left(\frac{t}{\rho(t)}\right)^d.
\end{equation}
Let $\ell=\rho(t)=t^{d/(d+m)}$. Then, $t=\ell^{(d+m)/d}$, and it follows from \eqref{eq120} that for all large $\ell$,
\begin{equation}
n_t\leq (3\sqrt{d})^d \ell^m \leq \ell^{m+1}.
\end{equation}
Now, with the choices $\ell=t^{d/(d+m)}$, $n=m+1$, and $x_j=z_j$ for $1\leq j\leq n_t$, where $(z_j:1\leq j\leq n_t)$ are as above, in view of \eqref{eq11} and since $\ell\to\infty$ as $t\to\infty$, Lemma~\ref{lemma1} implies the following. On a set of full $\mathbb{P}$-measure, there exists $t_0>0$ such that for all $t\geq t_0$, $B(x,\rho(t))$ contains a clearing of radius $R_{\rho(t)}$ for each $x\in C(0,t)$, where $R_{\rho(t)}$ is as in \eqref{eqrell}. That is, $\mathbb{P}(\Omega_1)=1$, where
\begin{equation} \label{eqenviron}
\Omega_1:=\{\omega\in\Omega:\exists\:t_0\:\:\forall\:t\geq t_0,\:\:\forall\:x\in C(0,t)\:\: \exists\:y\in B(x,\rho(t))\:\:\text{with}\: B\left(y,R_{\rho(t)}\right)\subseteq K^c\}. 
\end{equation}
It follows from \eqref{eqrell} and the requirement (c) that $2r(t)\leq R_{\rho(t)}$, and therefore for each $\omega\in\Omega_1$, for all large $t$ there is a clearing of radius $2r(t)$ inside any ball of radius $\rho(t)$ centered within $C(0,t)$. In this proof, we will use $\Omega_1$ as the almost-sure, i.e., \emph{quenched}, environment for the BM.  

Next, suppose that $t/h(t)$ is an integer for notational convenience\footnote{We would like to avoid the floor function in notation.}, and split $[0,t]$ into $t/h(t)$ pieces as 
$$[0,h(t)],\: [h(t),2h(t)],\:\ldots\:,[t-h(t),t].$$
For $j=1,2,\ldots,t/h(t)$, let
$$ x_j=X((j-1)h(t)),  $$
and define the intervals $I_{j,t}$ and the balls $B_{j,t}$, respectively, as
$$I_{j,t}=[(j-1)h(t),j h(t)],\quad B_{j,t}=B(x_j,\rho(t)).$$
We call $x\in\mathbb{R}^d$ a \emph{good point} for $\omega\in\Omega$ at time $t$ if $B(x,r(t))$ is a clearing (see Definition~\ref{def1}) in the random environment $\omega$. That is,
$$ \Phi_t^\omega=\{x\in\mathbb{R}^d:B(x,r(t))\subseteq K^c(\omega)\} $$
is the set of good points associated to the pair $(\omega,t)$. We now estimate the conditional probability that $X$ does not hit $\Phi_t^\omega$ up to a large time $t$ given that $\omega\in\Omega_1$. 

For $f\in C[0,t]$ and $0\leq a<b\leq t$, let $f_{[a,b]}=\{f(s):a\leq s\leq b\}$. Then, for $t>1$ and $1\leq j\leq t/h(t)$, define the events
$$E_{j,t}=\{X_{I_{j,t}}\cap \Phi_t^\omega=\emptyset\}, \quad G_{j,t}=\{|x_{j+1}-x_j|>h(t)\},$$
and let $E_t=\bigcap_{j=1}^{t/h(t)}E_{j,t}$. In words, $E_t$ is the event that $X$ does not hit a good point associated to $(\omega,t)$ over $[0,t]$, that is,
$$ E_t=\bigg\{\bigg(\bigcup_{0\leq s\leq t} \{X(s)\} \bigg)\cap \Phi_t^\omega=\emptyset\bigg\} . $$
Since a BM typically moves a distance of order $\sqrt{s}$ over a time period of length $s$, $G_{j,t}$ is an unlikely event for large $t$. Set $\mathbf{P}^\omega=\mathbf{P}^\omega_0$. Using that $E_t=\bigcap_{j=1}^{t/h(t)}E_{j,t}$, we estimate $\mathbf{P}^\omega(E_t)$ by applying repeated conditioning on $G_{j,t}^c$ at times $jh(t)$ for $1\leq j\leq t/h(t)-1$ and throwing away the rare events $G_{j,t}$ as
\begin{equation} \label{eq121}
\mathbf{P}^\omega(E_t)\leq \mathbf{P}^\omega(E_{1,t})\prod_{j=2}^{t/h(t)} \mathbf{P}^\omega\left(E_{j,t} \bigg| \bigcap_{k=1}^{j-1}(E_{k,t},G_{k,t}^c)\right)+\mathbf{P}^\omega(G_{1,t})+\sum_{j=2}^{t/h(t)}\mathbf{P}^\omega\left(G_{j,t} \bigg| \bigcap_{k=1}^{j-1}(E_{k,t},G_{k,t}^c) \right) .
\end{equation}
In the rest of the proof, we find a suitable upper bound on the right-hand side of \eqref{eq121}. 

Recall that $I_{j,t}=[(j-1)h(t),jh(t)]$ and $B_{j,t}=B(x_j,\rho(t))$ with $x_j=X((j-1)h(t))$. Let $q_0(t)$ be the probability that $X$ stays inside $B_{j,t}$ over the period $I_{j,t}$. From Proposition B, we have
\begin{equation} \label{eq122}
q_0(t)=\exp\left[-\frac{\lambda_d\,h(t)}{\rho^2(t)}(1+o(1))\right], \quad t\rightarrow\infty.   
\end{equation}
Let $\omega\in\Omega_1$ (see \eqref{eqenviron}), and choose $t$ large enough so that $t\geq t_0= t_0(\omega)$, where $t_0$ is as in \eqref{eqenviron}. Then, $B_{1,t}=B(0,\rho(t))$ contains a clearing of radius $2r(t)$, hence a ball of radius $r(t)$, say $\mathcal{B}_{1,t}$, that is entirely contained in $\Phi_t^\omega$. That is, $\mathcal{B}_{1,t}\subseteq \Phi_t^\omega \cap B_{1,t}$. Likewise, for each $j=2,\ldots,t/h(t)$, conditional on $\bigcap_{k=1}^{j-1}(E_{k,t},G_{k,t}^c)$, we have $|x_j|=|X((j-1)h(t))|<t$, and therefore by definition of $\Omega_1$, $B_{j,t}=B(x_j,\rho(t))$ contains a ball of radius $r(t)$, say $\mathcal{B}_{j,t}$, that is entirely contained in $\Phi_t^\omega$; that is, $\mathcal{B}_{j,t}\subseteq \Phi_t^\omega\cap B_{j,t}$. Now let $q_{1,j}(t)$ be the probability that $X$ doesn't hit $\mathcal{B}_{j,t}$ conditional on exiting $B_{j,t}$ over $I_{j,t}$. If $X$ is conditioned to exit $B_{j,t}=B(x_j,\rho(t))$ over $I_{j,t}$, over this same period it must also exit $B(x_j,\widehat{r}(t))$, where $\widehat{r}(t)$ is the distance between $x_j$ and the center of $\mathcal{B}_{j,t}$. Therefore, since the Brownian exit distribution out of a ball centered at the starting point has rotational invariance (even under the conditioning), by comparing the surface area of the $\widehat{r}(t)$-ball that intersects $\mathcal{B}_{j,t}$ to the total surface area of the $\widehat{r}(t)$-ball, and since $\widehat{r}(t)\leq \rho(t)$ for each $t>1$, we obtain   
\begin{equation} \label{eq123}
q_{1,j}(t)\leq 1-\frac{\kappa_d\, r^{d-1}(t)}{\rho^{d-1}(t)}=:q_1(t) \quad \text{for all $t>1$},
\end{equation} 
where $\kappa_d$ is a constant that only depends on the dimension $d$. Also, from Proposition A, we have
\begin{equation}  \label{eq124}
\mathbf{P}^\omega(G_{j,t})=\mathbf{P}^\omega(G_{1,t})\leq \mathbf{P}^\omega\left(\sup_{0\leq s\leq h(t)}|X(s)|>h(t)\right)\leq \exp\left[-\frac{1}{2}h(t)(1+o(1))\right].
\end{equation}
Now apply the Markov property of the Brownian path $(X(s))_{0\leq s\leq t}$ at times $h(t),2h(t),\ldots,t-h(t)$, and use \eqref{eq122}-\eqref{eq124} to continue the estimate in \eqref{eq121} as 
\begin{align} \label{eq125}
\mathbf{P}^\omega(E_t)&\leq [q_0(t)+q_1(t)]^{t/h(t)}+\frac{t}{h(t)}\mathbf{P}^\omega(G_{1,t})  \nonumber \\
&\leq \left[e^{-\frac{\lambda_d\,h(t)}{\rho^2(t)}(1+o(1))}+1-\frac{\kappa_d\, r^{d-1}(t)}{\rho^{d-1}(t)}\right]^{t/h(t)}+\frac{t}{h(t)}\exp\left[-\frac{1}{2}h(t)(1+o(1))\right] .
\end{align}
We now choose $h(t)$, $\rho(t)$, and $r(t)$ in a suitable way so as to keep $\mathbf{P}^\omega(E_t)$ sufficiently small in view of \eqref{eq125}, while respecting the previously stated requirements (a)-(c):  
$$ h(t)=\sqrt{t},\quad \rho(t)=t^{1/(6d)},\quad r(t)=\frac{R_0}{3}\left(\frac{1}{6d}\right)^{1/d}(\log t)^{1/d}, \quad t>1.$$
With these choices, since $\exp\left[-\frac{\lambda_d\,h(t)}{\rho^2(t)}\right] \leq \frac{(\kappa_d/2)r^{d-1}(t)}{\rho^{d-1}(t)}$ for all large $t$, it follows from \eqref{eq125} and the estimate $1+x\leq e^x$ that for all large $t$,
\begin{align} 
\mathbf{P}^\omega(E_t)&\leq \left[1-\frac{(\kappa_d/2)r^{d-1}(t)}{\rho^{d-1}(t)}\right]^{t/h(t)}+\frac{t}{h(t)}\exp\left[-\frac{1}{2}h(t)(1+o(1))\right] \nonumber \\
&\leq \exp\left[-\frac{\kappa_d}{2}\left(\frac{r(t)}{\rho(t)}\right)^{d-1}\frac{t}{h(t)}\right]+\frac{t}{h(t)}\exp\left[-\frac{1}{2}h(t)(1+o(1))\right] \nonumber \\
&=\exp\left[-\frac{\kappa_d}{2}\frac{c(R_0,d)(\log t)^{(d-1)/d}t^{1/2}}{t^{(d-1)/(6d)}}\right]+t^{1/2}\exp\left[-\frac{1}{2}t^{1/2}(1+o(1))\right] \nonumber \\ 
&\leq \exp\left[-t^{1/2-(d-1)/(6d)}\right]+\exp\left[-\frac{1}{3}t^{1/2}\right]\leq e^{-t^{1/3}}, \nonumber
\end{align}
where the last inequality follows since $\frac{1}{3}<\frac{1}{2}-\frac{d-1}{6d}$. Hence, we reach the following conclusion. There exists $\Omega_1\subseteq \Omega$ with $\mathbb{P}(\Omega_1)=1$ such that $\forall\,\omega\in\Omega_1$, $\exists\,t_0=t_0(\omega)$ such that $\forall\,t\geq t_0$,
\begin{equation} 
\mathbf{P}^\omega\left((\cup_{0\leq s\leq t}\{X(s)\})\cap \Phi_t^\omega=\emptyset\right)\leq e^{-t^{1/3}}. \nonumber
\end{equation}
\end{proof}

Next, we use Lemma~\ref{lemma2} to complete the first part of the proof of the upper bound of Theorem~\ref{thm2}. Recall that $0\leq \delta<\beta_2$. Choose $\alpha$ such that $0<\alpha<1-\delta/\beta_2$. Split the interval $[0,t]$ into two pieces as $[0,\alpha t]$ and $[\alpha t,t]$. We argue that with `high' probability, the BBM hits a good point, say $z_0\in\mathbb{R}^d$, associated to $(\omega,\alpha t)$ over $[0,\alpha t]$, and then the sub-BBM emanating from the particle that hits $z_0$ produces at least $e^{\delta t}$ particles over $[\alpha t,t]$ inside $B(z_0,r(\alpha t))$.

Let $Y_1=(Y_1(s))_{s\geq 0}$ be a randomly\footnote{One can choose an ancestral line \emph{randomly} as follows: start with the path of the initial particle from $t=0$ until it branches, and when it branches, pick one of the two offspring with probability $1/2$, and concatenate the previously traced path to the path of the chosen offspring until it, too, branches; repeat indefinitely the procedure of picking an offspring particle with probability $1/2$ upon branching and concatenating its path to the previously traced path.} chosen ancestral line of the BBM in the random environment $\omega$. Note that even under $P^\omega$, since branching and motion mechanisms are independent of each other, $(Y_1(s))_{s\geq 0}$ is identically distributed as a standard Brownian motion. The \emph{range} (accumulated support) of $Z$ is the process defined by 
\begin{equation} \label{range}
\mathcal{R}(t)=\bigcup_{0\leq s\leq t} \text{supp}(Z(s)).
\end{equation}
Since $Y_1$ is an ancestral line of $Z$, we have $\cup_{0\leq s\leq t}Y_1(s)\subseteq \mathcal{R}(t)$ for each $t\geq 0$. Then, since $Y_1$ is Brownian, Lemma~\ref{lemma2} implies that for $0<\alpha<1$, on a set of full $\mathbb{P}$-measure, say $\Omega_2$, for all large $t$,
\begin{equation} \label{eq208}
P^\omega(R(\alpha t)\cap \Phi_{\alpha t}^\omega=\emptyset)\leq e^{-\alpha^{1/3}t^{1/3}}.
\end{equation}
Observe that $\{\mathcal{R}(\alpha t)\cap \Phi_{\alpha t}^\omega=\emptyset\}$ is the event that $Z$ doesn't hit a good point associated to $(\omega,\alpha t)$ over $[0,\alpha t]$. 

Now let $\tau=\tau(\omega)=\inf\{s>0:\mathcal{R}(s)\cap \Phi_{\alpha t}^\omega\neq\emptyset\}$ be the first hitting time of $Z$ to $\Phi_{\alpha t}^\omega$. Let $Y_2$ be the ancestral line of $Z$ that first hits $\Phi_{\alpha t}^\omega$, and let $z_0=Y_2(\tau)$. Conditional on $\tau<\alpha t$, apply the strong Markov property at time $\tau$, and then apply Theorem~\ref{thm1} to the growth inside $B(z_0,r(\alpha t))$ of the sub-BBM initiated by $Y_2$ at time $\tau$. Note that $\mathcal{B}_t:=B(z_0,r(\alpha t))$ is a clearing in the random environment $\omega$ by definition of $\tau$, $z_0$ and $\Phi_{\alpha t}^\omega$. 

In detail, for $u\geq 0$, let $\big|Z_{[\tau,\tau+u]}^{\mathcal{B}_t}\big|$ denote the mass at time $\tau+u$ of the sub-BBM initiated at position $z_0$ and time $\tau$ by $Y_2$ with deactivation at $\partial\mathcal{B}_t$. Let $s:=(1-\alpha)t$, $\hat{r}:\mathbb{R}_+\to\mathbb{R}_+$ be such that
$$ \hat{r}(s)=\frac{R_0}{3}\left(\frac{1}{6d}\right)^{1/d}\left[\log\left(\frac{\alpha s}{1-\alpha}\right)\right]^{1/d} $$
for large $s$, and $B_s:=B(0,\hat{r}(s))$. Observe the equality of events $\{\tau\leq \alpha t\}=\{\mathcal{R}(\alpha t)\cap\Phi_{\alpha t}^\omega\neq\emptyset\}$, and that $t-\tau\geq (1-\alpha)t$ conditional on $\tau\leq \alpha t$, and $\hat{r}(s)=r(\alpha t)$. Then, on $\Omega_2$ with $\mathbb{P}(\Omega_2)=1$, applying the strong Markov property at $\tau=\tau(\omega)$, and taking $\gamma_s=\exp[-\sqrt{\beta_2/2}\,\hat{r}(s)]$ for instance, Theorem~\ref{thm1} implies that for all large $t$,
\begin{align} \label{eq209}
P^\omega\left(N_t<e^{\delta t} \big\vert R(\alpha t)\cap \Phi_{\alpha t}^\omega\neq\emptyset \right)&\leq 
P^\omega\left(\big | Z_{[\tau,t]}^{\mathcal{B}_t} \big |< e^{-\sqrt{\beta_2/2}\,\hat{r}(s)} e^{-\frac{\lambda_d s}{(\hat{r}(s))^2}(1+o(1))}e^{\beta_2 s}\right) \nonumber \\
& \leq  P^\omega\left(\big | Z_{[\tau,\tau+(1-\alpha)t]}^{\mathcal{B}_t} \big |< e^{-\sqrt{\beta_2/2}\,\hat{r}(s)} e^{-\frac{\lambda_d s}{(\hat{r}(s))^2}(1+o(1))}e^{\beta_2 s}\right) \nonumber \\
& = P\left(\big | Z_s^{B_s} \big |< e^{-\sqrt{\beta_2/2}\,\hat{r}(s)} e^{-\frac{\lambda_d s}{(\hat{r}(s))^2}(1+o(1))}e^{\beta_2 s}\right) \nonumber \\
&=e^{-\sqrt{\beta_2/2}\,\hat{r}(s)(1+o(1))},
\end{align} 
where, in the first inequality, we have used that $\delta t<(1-\alpha)\beta_2 t=\beta_2 s$ due to the choice $\alpha<1-\delta/\beta_2$, and in the first equality we have used that $\mathcal{B}_t$ is a clearing in $\omega$ followed by translation invariance. Then, in view of $\hat{r}(s)=r(\alpha t)$ and the definition of $r(t)$ from \eqref{part1eqradius}, we reach the following conclusion via \eqref{eq208} and \eqref{eq209}: on $\Omega_2\subseteq\Omega$ with $\mathbb{P}(\Omega_2)=1$,
\begin{equation} \label{eq210}
P^\omega\left(N_t<e^{\delta t}\right)\leq e^{-c(\log t)^{1/d}(1+o(1))},
\end{equation}
where $c=c(d,\nu,\beta_2,\delta)>0$. (The dependence of $c$ on $\nu$ is through $R_0$, which appears in the definition of $r(t)$; see \eqref{eqr0} and \eqref{part1eqradius}.) This gives a quenched upper bound on the probability that $N_t=|Z_t|$ is exponentially few, and completes the first part of the proof of the lower bound of Theorem~\ref{thm2}. 

\medskip

\noindent \textbf{\underline{Part 2}: Time scales within $[0,t]$ and moving a particle into a large clearing}

This part of the proof is not new; it is essentially taken from \cite{E2008} with minor improvements, where we also estimate the rate of decay to zero as $t\rightarrow\infty$ of the probabilities of the relevant unlikely events as opposed to merely showing that they tend to zero. Introduce two different time scales, $\ell(t)$ and $m(t)$, where $\ell(t)=o(m(t))$ and $m(t)=o(t)$, and split the interval $[0,t]$ into $[0,\ell(t)]$, $[\ell(t),m(t)]$ and $[m(t),t]$. More precisely, let $\ell,m:\mathbb{R}_+\to\mathbb{R}_+$ be two functions satisfying $\ell(t)< m(t)< t$ for all $t> 0$, and
\begin{enumerate}
	\item[(i)] $\lim_{t\rightarrow\infty}\ell(t)=\infty$,
	\item[(ii)] $\lim_{t\rightarrow\infty} \frac{\log t}{\log (\ell(t))}=1$,
	\item[(iii)] $\ell(t)=o(m(t))$,
	\item[(iv)] $m(t)=o(\ell^2(t))$,
	\item[(v)] $m(t)=o(t(\log t)^{-2/d})$.
\end{enumerate}
For concreteness, we fix the following choices of $\ell$ and $m$ that satisfy (i)-(v): let $\ell(t)$ and $m(t)$ be arbitrarily defined with $\ell(t)< m(t)< t$ for $t\in(0,e]$, and
$$\ell(t)=t^{1-1/(\log\log t)},\quad m(t)=t^{1-1/(2\log\log t)},\quad \text{for}\:\: t>e .$$ 

Firstly, using Part 1 of the proof, we prepare the setting at time $\ell(t)$.
Fix $\delta\in(0,\beta_2)$, and define
$$I(t)=\lfloor e^{\delta \ell(t)} \rfloor .$$
Recall the definition of $\mathcal{R}(t)$ from \eqref{range}, and for $t>0$, let
$$M_t:=\inf \{r\geq 0:\mathcal{R}(t)\subseteq B(0,r)\}.$$
Next, for $t>0$, define the families of events
\begin{equation*}
G_t:=\{N_{\ell(t)}\geq I(t)\},\quad H_t:=\{M_{\ell(t)}\leq\sqrt{2\beta_2+\varepsilon}\,\ell(t)\}. 
\end{equation*}
Recall that we write $Z_t(B)$ to denote the mass of $Z$ that fall inside $B$ at time $t$, and define further the family of events  
\begin{equation*}
F_t:=\left\{Z_{\ell(t)}\left(B(0,\sqrt{2\beta_2+\varepsilon}\,\ell(t))\right)\geq I(t)\right\}. 
\end{equation*}
Since $\lim_{t\rightarrow\infty}\ell(t)=\infty$, by \eqref{eq210}, on a set of full $\mathbb{P}$-measure, which we had called $\Omega_2$,
\begin{equation}\label{eq301}
P^\omega(G_t^c)=P^\omega\left(N_{\ell(t)}<\lfloor e^{\delta \ell(t)} \rfloor\right)\leq e^{-c(\log \ell(t))^{1/d}(1+o(1))}
\end{equation}
with $c=c(d,\nu,\beta_2,\delta)>0$. Next, we establish some control on the spatial spread of the BBM at time $\ell(t)$. Observe that $M_t/t$ is a measure of the spread of $Z$ over $[0,t]$. As before, let $\mathcal{N}_t$ denote the set of particles of $Z$ that are alive at time $t$, and for $u\in\mathcal{N}_t$, let $(Y_u(s))_{0\leq s\leq t}$ denote the ancestral line up to $t$ of particle $u$. Note that $N_t=|\mathcal{N}_t|$. Then, using the union bound, for $\gamma>0$,
\begin{equation} \label{eq302}
P\left(M_t>\gamma t\right)= P\left(\exists u\:\in \mathcal{N}_t:\sup_{0\le s\le t}|Y_u(s)|>\gamma t\right) \le E[N_t]\:\mathbf{P}_0\left(\sup_{0\le s\le t}|X(s)|>\gamma t\right). 
\end{equation}
Here, we use $P$ for the law of a standard BBM in $\mathbb{R}^d$ with constant binary branching rate $\beta_2$ everywhere. It is a standard result that $E[N_t]=\exp(\beta_2 t)$ (one can deduce this, for example, from Proposition C), and we know from Proposition A that $\mathbf{P}_0\left(\sup_{0\le s\le t}|X(s)|>\gamma t\right)=\exp[-\gamma^2 t/2(1+o(1))]$. Moreover, the following stochastic domination is clear: for all $B\subseteq\mathbb{R}^d$ Borel, all $k\in\mathbb{N}$, and $t\geq 0$,
\begin{equation} \label{eq303} 
P(Z_t(B)<k)\leq P^\omega(Z_t(B)<k) \quad \text{for each $\omega\in\Omega$}.
\end{equation}
Then, taking $B=\left(B(0,\sqrt{2\beta+\varepsilon}\,\ell(t))\right)^c$ and $k=1$ in \eqref{eq303}, and $\gamma=\sqrt{2\beta_2+\varepsilon}$ in \eqref{eq302}, and replacing $t$ by $\ell(t)$ in both \eqref{eq302} and \eqref{eq303}, it follows that on $\Omega$, for any $\varepsilon>0$,
\begin{equation} \label{eq304}
P^\omega(H_t^c)=P^\omega\left(M_{\ell(t)}>\sqrt{2\beta_2+\varepsilon}\,\ell(t)\right)\leq e^{-\frac{\varepsilon \ell(t)}{2}(1+o(1))}.
\end{equation} 
Since $G_t\cap H_t \subseteq F_t$, we have $P^\omega(F_t^c)\leq P^\omega(G_t^c)+P^\omega(H_t^c)$, which, in view of \eqref{eq301} and \eqref{eq304} implies that on $\Omega_2$,
\begin{equation} \label{firstestimate}
P^\omega(F_t^c)\leq e^{-c(\log \ell(t))^{1/d}(1+o(1))}, \quad\quad c=c(d,\nu,\beta_2,\delta)>0.
\end{equation} 
This means, on a set of full $\mathbb{P}$-measure, with `high' $P^\omega$-probability, there are at least $I(t)$ particles in $B(0,\sqrt{2\beta_2+\varepsilon}\,\ell(t))$ at time $\ell(t)$ for large $t$.

Next, we prepare the setting at time $m(t)$. Recall \eqref{eqrell} and define 
\begin{equation} \label{eqbigradius}
R(t)=R_{\ell(t)}=R_0[\log \ell(t)]^{1/d}-[\log\log\ell(t)]^2, \quad\text{for}\:\:t>e^e.
\end{equation}
Since $\lim_{t\rightarrow\infty}\ell(t)=\infty$, Lemma~\ref{lemma1} implies that on a set of full $\mathbb{P}$-measure, say $\Omega_3$, there is a clearing $B(x_0,R(t)+1)$ such that $|x_0|\leq\ell(t)$ for all large $t$. Let $\omega\in\Omega_2\cap\Omega_3$. Conditional on the event $F_t$, the distance between $x_0$ and each of the at least $I(t)$ many particles in $B(0,\sqrt{2\beta_2+\varepsilon}\,\ell(t))$ at time $\ell(t)$ is at most
$$(1+\sqrt{2\beta_2+\varepsilon})\ell(t) .$$ 
A Brownian particle present at time $\ell(t)$ in $B(0,\sqrt{2\beta_2+\varepsilon}\,\ell(t))$ moves to $B(x_0,1)$ over $[\ell(t),m(t)]$ with probability at least
$$q_t=\exp\left[-\frac{[(1+\sqrt{2\beta_2+\varepsilon})\ell(t)]^2}{2[m(t)-\ell(t)]}(1+o(1))\right], $$
which follows from (iii) and (iv), along with the Brownian transition density. Apply the Markov property of the BBM at time $\ell(t)$, and neglect possible branching of particles over $[\ell(t),m(t)]$ for an upper bound on the probability of $C_t^c$, where
\begin{equation}
C_t:=\{Z_{m(t)}(B(x_0,1))>0\}. \label{eqct}
\end{equation}
Observe that conditional on the event $F_t$, the event $C_t$ is realized if one of the sub-BBMs initiated by each of the at least $I(t)$ many particles present in $B(0,\sqrt{2\beta_2+\varepsilon}\,\ell(t))$ at time $\ell(t)$ contributes a particle to $B(x_0,1)$ at time $m(t)$. Therefore, by the independence of particles present at time $\ell(t)$, we have
\begin{equation} \label{eq305}
P^\omega(C_t^c \mid F_t)\leq (1-q_t)^{I(t)}=e^{-q_t I(t)},
\end{equation}
where we have used the estimate $1+x\leq e^x$. Since (iii) implies that $\frac{\ell^2(t)}{m(t)}=o(\ell(t))$, we have
$$q_t e^{\delta \ell(t)}=\exp[\delta\ell(t)(1+o(1))].$$
Then, it follows from \eqref{eq305} that for all large $t$,
\begin{equation} \label{eq306}
P^\omega(C_t^c \mid F_t)\leq \exp\left[-e^{\delta \ell(t)(1+o(1))}\right] \leq e^{-t^2},
\end{equation}
where we have used that by choice, $\ell(t)=t^{1-1/(\log\log t)}$. Note that \eqref{eq306} is superexponentially small in $t$. Thus far, the value of $\delta\in(0,\beta_2)$ was arbitrary. For the rest of the argument, the exact value of $\delta$ has no importance; therefore let us now fix it as $\delta=\beta_2/2$. Then, it follows from \eqref{firstestimate}, \eqref{eq306}, and assumption (ii) that on $\Omega_0:=\Omega_2 \cap \Omega_3$ with $\mathbb{P}(\Omega_0)=1$, 
\begin{equation} \label{secondestimate} 
P^\omega(C_t^c)\leq P^\omega(C_t^c \mid F_t) + P(F_t^c)\leq e^{-c(\log t)^{1/d}(1+o(1))}, \quad\quad c=c(d,\nu,\beta_2)>0.
\end{equation}
This means, on a set of full $\mathbb{P}$-measure, with `high' $P^\omega$-probability, there is at least one particle of $Z$ inside $B(x_0,1)$ at time $m(t)$ for large $t$, where $|x_0|\leq \ell(t)$. Let us generically call this particle $v$, and denote by $y_0:=X_v(m(t))$ its position at time $m(t)$.

\medskip

\noindent \textbf{\underline{Part 3}: BBM in the large expanding clearing}

Let $\omega\in\Omega_0$, and recall that $B(x_0,R(t)+1)$ is a clearing in $\omega$. In this part of the proof, we will work under the law $P^\omega(\:\cdot\: \mid C_t)$, where $C_t$ was defined in \eqref{eqct}. Conditional on the event $C_t$, we show that $B_t=B(y_0,R(t))$ is a large enough expanding clearing in $\omega$ in which the BBM can produce sufficiently many particles. In particular, we study the evolution of the sub-BBM initiated by $v$ at time $m(t)$ over the period $[m(t),t]$ within the expanding clearing $B_t$. Denote this sub-BBM by $\widehat{Z}$. We will use $P_x^\omega$ for the law of a BBM started with a single particle at $x\in\mathbb{R}^d$ in the random environment $\omega$. For $t>0$ and $\varepsilon>0$, define
\begin{equation} \label{newneweq}
A_{t,\varepsilon}:=\left\{N_t<\exp\left[t\left(\beta_2-\frac{c(d,v)+\varepsilon}{(\log t)^{2/d}}\right)\right] \right\}.
\end{equation}
Recall that our goal (see \eqref{goal}) is to find a suitable upper bound on $P^\omega(A_{t,\varepsilon})$. 

Define $\widehat{R}:\mathbb{R}_+\to\mathbb{R}_+$ such that $\widehat{R}(t-m(t))=R(t)$ for all large $t$. (By the choice of $m(t)$, $t-m(t)$ is increasing on $t\geq t_0$ for some $t_0>0$. Therefore, $t_1-m(t_1)=t_2-m(t_2)$ implies that $t_1=t_2$ for $t_1\wedge t_2\geq t_0$). Next, let $s:=t-m(t)$, $\widehat{B}_s:=B(y_0,\widehat{R}(s))$ and
\begin{equation} 
p_s:=\mathbf{P}_{y_0}(\sigma_{\widehat{B}_s}\geq s)=\mathbf{P}_{0}(\sigma_{B(0,\widehat{R}(s))}\geq s), \nonumber
\end{equation}    
where, as before, $\mathbf{P}_x$ denotes the law of a standard BM started at $x\in\mathbb{R}^d$. By the Markov property of $Z$ applied at time $m(t)$, $\widehat{Z}$ is a BBM started with a single particle at $y_0$. Noting that $\widehat{B}_s$ is a clearing in $\omega$ for all large $s$, and taking $\gamma_s=\exp[-\sqrt{\beta_2/2}\,\widehat{R}(s)]$, Theorem~\ref{thm1} implies that 
\begin{align}  
P^\omega\left(|\widehat{Z}_s|< e^{-\sqrt{\beta_2/2}\,\widehat{R}(s)} p_s e^{\beta_2 s} \:\big|\: C_t \right)&\leq P_{y_0}^\omega\left(\big| Z_s^{\widehat{B}_s}\big|< e^{-\sqrt{\beta_2/2}\,\widehat{R}(s)} p_s e^{\beta_2 s}\right) \nonumber \\
&=\exp\left[-\sqrt{\beta_2/2}\,\widehat{R}(s)(1+o(1))\right]. \label{eq309}
\end{align}  
By Proposition B, \eqref{eqbigradius}, and since $\widehat{R}(s)=R(t)$ and $c(d,v)=\lambda_d/R_0^2$,  
\begin{equation} \label{eq310}
p_s=\exp\left[-\frac{\lambda_d s}{\widehat{R}^2(s)}(1+o(1))\right]=\exp\left[-\frac{c(d,v)(t-m(t))}{(\log \ell(t))^{2/d}}(1+o(1))\right]. 
\end{equation}
For two functions $f,g:\mathbb{R}_+\to\mathbb{R}_+$, use $f(t)\sim g(t)$ to express that $\lim_{t\rightarrow\infty}f(t)/g(t)=1$. Then, it follows from assumptions $(ii)$ and $(v)$ that
$$\frac{t-m(t)}{(\log \ell(t))^{2/d}} \sim \frac{t}{(\log t)^{2/d}},$$
by which, we can continue \eqref{eq310} with
\begin{equation} \label{eq311}
p_s=\exp\left[-\frac{c(d,v) t}{(\log t)^{2/d}}(1+o(1))\right].
\end{equation}
Furthermore, using that $s=t-m(t)$, we have for any $\varepsilon>0$,
\begin{equation} \nonumber
\exp\left[t\left(\beta_2-\frac{c(d,v)+\varepsilon}{(\log t)^{2/d}}\right)\right]=\exp\left[\beta_2 s+\beta_2 m(t)-\frac{(c(d,v)+\varepsilon)t}{(\log t)^{2/d}}\right] \leq e^{-\sqrt{\beta_2/2}\,\widehat{R}(s)} p_s e^{\beta_2 s}   
\end{equation}
for all large $t$, where we have used \eqref{eq311}, assumption (v), and that $\widehat{R}(s)=R(t)=o(t(\log t)^{-2/d})$ in passing to the inequality. It is clear that $N_t\geq |\widehat{Z}_{t-m(t)}|=|\widehat{Z}_s|$. Then, it follows from \eqref{eq309} that for all large $t$,
\begin{align} 
P^\omega\left(N_t<\exp\left[t\left(\beta_2-\frac{c(d,v)+\varepsilon}{(\log t)^{2/d}}\right)\right] \:\bigg|\:C_t\right)&\leq P^\omega\left(|\widehat{Z}_s|<e^{-\sqrt{2\beta_2}\widehat{R}(s)} p_s e^{\beta_2 s}\:\big|\:C_t\right)   \nonumber \\
&\leq \exp\left[-\sqrt{2\beta_2}\widehat{R}(s)(1+o(1))\right].  \nonumber
\end{align}
Finally, using assumption (ii), \eqref{eqbigradius} along with $\widehat{R}(s)=R(t)$, and \eqref{secondestimate}, we reach the following conclusion. On $\Omega_0$, which is a set of full $\mathbb{P}$-measure, for any $\varepsilon>0$,
\begin{equation} \label{eq313}
P^\omega\left(N_t<\exp\left[t\left(\beta_2-\frac{c(d,v)+\varepsilon}{(\log t)^{2/d}}\right)\right]\right)\leq \exp\left[-c(\log t)^{1/d}(1+o(1))\right],
\end{equation}
where $c=c(d,\nu,\beta_2)>0$. In the next part of the proof, we will exploit the fact that the right-hand side of \eqref{eq313} does not depend on $\varepsilon$.

\medskip

\noindent \textbf{\underline{Part 4}: Borel-Cantelli argument} 

We will show that on a set of full $\mathbb{P}$-measure, for any $\varepsilon>0$,
\begin{equation} \label{eqlowerbound}
\underset{t\rightarrow\infty}{\liminf}\, (\log t)^{2/d}\left(\frac{\log N_t}{t}-\beta_2\right)\geq -\left[c(d,v)+\varepsilon\right] \quad P^\omega\text{-a.s.}
\end{equation}
Recall the definition of $A_{t,\varepsilon}$ from \eqref{newneweq}. It follows from \eqref{eq313} that there exists $c=c(d,\nu,\beta_2)/2$, independent of $\varepsilon$, such that on $\Omega_0$, for all large $t$,
\begin{equation} \label{eq314}
P^\omega\left(A_{t,\varepsilon/2}\right)\leq e^{-c(\log t)^{1/d}}.  
\end{equation}
Define the function $f:\mathbb{N}\to\mathbb{R}_+$ by 
\begin{equation}\label{eqfk}
f(k)=\exp\left[\left(\frac{2}{c}\right)^d(\log k)^d\right].
\end{equation}
Take $\omega\in\Omega_0$. By the choice of $f(k)$ and \eqref{eq314}, there exist constants $c_0>0$ and $k_0>0$ such that
\begin{equation*}
\sum_{k=1}^\infty P^\omega\left(A_{f(k),\varepsilon/2}\right)\leq c_0+\sum_{k=k_0}^\infty e^{-c(\log f(k))^{1/d}} = c_0+\sum_{k=k_0}^\infty\frac{1}{k^2}<\infty.
\end{equation*}
Then, by the Borel-Cantelli lemma, on a set of full $P^\omega$-measure, only finitely many events $A_{f(k),\varepsilon/2}$ occur. That is, for each $\omega\in\Omega_0$ there exists $\widehat{\Omega}_0\subseteq\widehat{\Omega}$ such that 
\begin{equation} \label{eq315}
P^\omega(\widehat{\Omega}_0)=1,\quad \widehat{\Omega}_0=\left\{\varpi\in\widehat{\Omega}:\exists\:k_0(\varpi)\:\:\forall\,k\geq k_0\:\:N_{f(k)}\geq \exp\left[f(k)\left(\beta_2-\frac{c(d,v)+\varepsilon/2}{(\log f(k))^{2/d}}\right)\right] \right\} .
\end{equation}
To prove \eqref{eqlowerbound}, it suffices to show that for each $\varpi\in\widehat{\Omega}_0$, 
\begin{equation} \label{eq316}
N_s \geq \exp\left[s\left(\beta_2-\frac{c(d,v)+\varepsilon}{(\log s)^{2/d}}\right)\right],\quad f(k)<s<f(k+1)
\end{equation} 
for all large $k$. Indeed, \eqref{eq315} and \eqref{eq316} would together imply \eqref{eqlowerbound} on $\Omega_0$. Observe that $N_s$ is $P^\omega$-almost surely increasing in $s$, and the right-hand side of \eqref{eq316} is also increasing in $s$ for all large $s$. Therefore, to prove \eqref{eq316}, it suffices to show that for all large $k$,
\begin{equation} \label{eq317}
N_{f(k)} \geq \exp\left[f(k+1)\left(\beta_2-\frac{c(d,v)+\varepsilon}{(\log f(k+1))^{2/d}}\right)\right].
\end{equation}
Next, we control $f(k+1)-f(k)$. Using \eqref{eqfk}, it can be shown that as $k\rightarrow\infty$,
\begin{equation} \label{eqfk2}
f(k+1)-f(k) \sim d\left(\frac{2}{c}\right)^d f(k)\frac{(\log k)^{d-1}}{k},
\end{equation} 
and that if we set $t=f(k)$, then
\begin{equation} \label{eqfk3}
k=\exp\left[\frac{c}{2}(\log t)^{1/d}\right].
\end{equation}
Then, setting $g(t)=f(k+1)-f(k)$, it follows from \eqref{eqfk2} and \eqref{eqfk3} that
\begin{equation}  \label{eqfk4}
g(t) \sim d\left(\frac{2}{c}\right)^d t\, \frac{(c/2)^{d-1}(\log t)^{(d-1)/d}}{\exp\left[\frac{c}{2}(\log t)^{1/d}\right]}.
\end{equation}
Using \eqref{eqbigradius} and \eqref{eqfk4}, it can be shown that
\begin{equation*} 
\underset{t\rightarrow\infty}{\lim}\frac{g(t) R^2(t)}{t}=0.
\end{equation*} 
This implies that for all large $k$,
\begin{equation*} 
f(k+1)-f(k)\leq \frac{\varepsilon f(k)}{2\beta_2 [\log f(k)]^{2/d}},
\end{equation*}
which further implies that
\begin{equation} \label{eq319}
\beta_2 f(k)-\frac{c(d,v)+\varepsilon/2}{[\log f(k)]^{2/d}}f(k) \geq \beta_2 f(k+1)-\frac{c(d,v)+\varepsilon}{[\log f(k+1)]^{2/d}}f(k+1)
\end{equation}
since $f(z)/[\log f(z)]^{2/d}$ is increasing for large $z$. As \eqref{eq319} implies \eqref{eq317} on $\widehat{\Omega}_0$, this proves \eqref{eqlowerbound}, and hence completes the proof of the lower bound of Theorem~\ref{thm2}.

\section{Proof of Theorem~\ref{thm1}}\label{section4}

Recall that by assumption $r:\mathbb{R}_+ \to \mathbb{R}_+$ satisfies $r(t)\to\infty$ as $t\to\infty$ and $r(t)=o(\sqrt{t})$. Also, we set $B_t=B(0,r(t))$ for $t\geq 0$. Throughout this section, we will use that the law of $|Z_t^{B_t}|$ is the same as the law of number of particles of $Z$ which are present at $t$ and whose ancestral lines over $[0,t]$ have been confined to $B_t$.  

\subsection{Proof of the lower bound}

Suppose that 
$$\gamma_t=e^{-\kappa r(t)}, \quad \text{where} \:\: \kappa>0.$$
Consider the joint strategy of suppressing the branching over $[0,f(t)]$, and then letting the BBM evolve `normally' in the remaining interval $[f(t),t]$. To be precise, recall that $n_t:=|Z_t^{B_t}|$, $\sigma_A$ denotes the first exit time out of $A$, and $p_t:=\mathbf{P}_0(\sigma_{B_t}\geq t)$; and let $f:\mathbb{R}_+\to\mathbb{R}_+$ satisfy $f(t)=o(t)$. For $t>0$ define the events
$$ A_t=\{N_{f(t)}=1\}, \quad E_t=\{n_t<\gamma_t p_t e^{\beta t} \}.$$
Estimate
\begin{equation} \label{eq20}
P(E_t)\geq P(E_t \cap A_t)=P(E_t \mid A_t) P(A_t).
\end{equation}
We will show that $P(E_t^c \mid A_t)$ tends to a constant smaller than one as $t\rightarrow\infty$ for suitable $f$. Let $(Y_1(s))_{0\leq s\leq \tau_1}$ be the path of the initial particle, where $\tau_1$ denotes the particle's lifetime. Conditional on $A_t$, it is clear that $\tau_1\geq f(t)$, and that $n_t=0$ if $Y_1(z)\notin B_t$ for some $0\leq z \leq f(t)$. Next, for $t>0$ define
$$ D_t=\{ Y_1(z)\in B_t\:\:\forall\, 0\leq z \leq f(t) \}  .$$
Then, 
\begin{equation} \label{eq21}
E\left[n_t \mid A_t\right]= E\left[n_t \mathbbm{1}_{D_t} \mid A_t\right] + E\left[n_t \mathbbm{1}_{D_t^c} \mid A_t\right]= E\left[n_t \mid A_t, D_t\right] P(D_t \mid A_t),
\end{equation}
where the second term on the right-hand side vanishes. Write
\begin{equation} \label{eq021}
E\left[n_t \mid A_t, D_t\right]= \int_{B_t} E\left[n_t \mid A_t, D_t, Y_1(f(t))=y\right]P(Y_1(f(t)) \in dy \mid A_t,D_t) .
\end{equation}
Define
\begin{equation} \label{neweq3}
\widetilde{p}^{(t)}(x,s,dy):=\mathbf{P}_x(X(s)\in dy \mid X(z)\in B_t \:\:\forall\, 0\leq z\leq s) \quad \text{and} \quad p_{s,x}^t:=\mathbf{P}_x(\sigma_{B_t}\geq s). 
\end{equation}
Applying the Markov property of a standard BM at time $s$ with $0<s<t$ gives
\begin{equation} \label{neweq4}
p_t=p_{s,0}^t \int_{B_t} p_{t-s,y}^t\, \widetilde{p}^{(t)}(0,s,dy).
\end{equation}
Furthermore, it follows from \eqref{eq21} and \eqref{eq021} that
\begin{equation} \label{eq22}
E\left[n_t \mid A_t\right]= p_{f(t),0}^t \int_{B_t} E\left[n_t \mid A_t, D_t, Y_1(f(t))=y\right] \widetilde{p}^{(t)}(0,f(t),dy) .
\end{equation} 
Now apply the Markov property of BBM at time $f(t)$, and use the many-to-one lemma (see for instance \cite[Lemma 1.6]{E2014}) to obtain
\begin{equation} \label{eq022}
E\left[n_t \mid A_t, D_t, Y_1(f(t))=y\right] = p^t_{t-f(t),y}\, e^{\beta(t-f(t))}  ,\quad y\in B_t . 
\end{equation}
Using \eqref{neweq4} with $s$ therein replaced by $f(t)$, it then follows from \eqref{neweq4}, \eqref{eq22} and \eqref{eq022} that
\begin{align} 
E\left[n_t \mid A_t\right]&= e^{\beta(t-f(t))} p_{f(t),0}^t \int_{B_t} p^t_{t-f(t),y}\, \widetilde{p}^{(t)}(0,f(t),dy) \nonumber \\
&=  e^{\beta(t-f(t))} p_t . \nonumber
\end{align}
Then, by the Markov inequality,   
\begin{equation} \label{eq23}
P(E_t^c \mid A_t)\leq \frac{E\left[n_t \big\vert A_t\right]}{\gamma_t p_t e^{\beta t}}=\gamma_t^{-1}e^{-\beta f(t)}.
\end{equation}
Choose $f(t)=-(1/\beta)\log ((1-\delta)\gamma_t)$, where $0<\delta<1$. With this choice of $f$, \eqref{eq23} implies that $P(E_t\mid A_t)\geq \delta$. Then, noting that $P(A_t)=e^{-\beta f(t)}$, it follows from \eqref{eq20} that
\begin{equation} 
P(E_t)\geq \delta e^{-\beta f(t)} = \delta(1-\delta)\gamma_t = e^{-\kappa r(t)(1+o(1))}.  \nonumber
\end{equation}
This, along with \eqref{eq30}, proves \eqref{eq2}, and the lower bound in \eqref{eq1}.

\subsection{Proof of the upper bound}

For the proof of the upper bound, we follow a method that is based on Chebyshev's inequality, similar to the proof of \cite[Thm.\ 1]{E2008}. Let $g:\mathbb{R}_+ \to \mathbb{R}_+$ satisfy $g(t)\to 0$ as $t\to\infty$. Later, we will choose $g_t:=g(t)$ in a precise way. For $t\geq 0$, let $N_t=|Z_t|$ as before, and estimate
\begin{equation}
P(\:\cdot\:)\leq P\left(\:\:\cdot\:\:\big | N_t>e^{\beta t}g_t\right)+P\left(N_t\leq e^{\beta t}g_t\right). \label{eq300}
\end{equation}
We first bound the second term on the right-hand side of \eqref{eq300} from above. It follows from \eqref{eq02} that $P(N_t\leq k)=1-(1-e^{-\beta t})^k\leq ke^{-\beta t}$ for $k\geq 1$. Set $k=\lfloor e^{\beta t}g_t \rfloor$ to obtain 
\begin{equation} \label{eq4}
P\left(N_t\leq e^{\beta t}g_t\right)=P\left(N_t\leq \lfloor e^{\beta t}g_t \rfloor\right)\leq \lfloor e^{\beta t}g_t \rfloor e^{-\beta t}\leq g_t.
\end{equation}
Next, for $t>0$ define 
$$\widetilde{P}_t(\:\cdot\:)=P(\:\cdot\: \mid N_t>e^{\beta t}g_t),$$
and let $\widetilde{E}_t$ be the corresponding expectation.
We now bound the first term on the right-hand side of \eqref{eq300} from above. Let $\mathcal{N}_t$ denote the set of particles of $Z$ that are alive at time $t$. For $u\in\mathcal{N}_t$, let $(Y_u(s))_{0\leq s\leq t}$ denote the ancestral line up to $t$ of particle $u$. By the \emph{ancestral line up to $t$} of a particle present at time $t$, we mean the continuous trajectory traversed up to $t$ by the particle, concatenated with the trajectories of all its ancestors including the one traversed by the initial particle. Note that $(Y_u(s))_{0\leq s\leq t}$ is identically distributed as a Brownian path $(X(s))_{0\leq s\leq t}$ for each $u\in\mathcal{N}_t$. Let us pick randomly, independent of their genealogy and position, $\lfloor e^{\beta t} g_t \rfloor$ particles from $\mathcal{N}_t$. Note that this is possible under $\widetilde{P}_t(\:\cdot\:)$. Denote this collection of particles by $\mathcal{M}_t$, set $M_t:=|\mathcal{M}_t|$, and define 
$$\hat{n}_t=\sum_{u\in \mathcal{M}_t} \mathbbm{1}_{A_u},$$
where $A_u=\{Y_u(s)\in B_t\:\:\forall\:0\leq s\leq t \}$. Observe that $\hat{n}_t$ counts, out of $\mathcal{M}_t$, the particles whose ancestral lines are confined to $B_t$ over $[0,t]$. Since the collection $\mathcal{M}_t$ is chosen independently of the motion process, each particle $u$ in $\mathcal{M}_t$ has an ancestral line $(Y_u(s))_{0\leq s\leq t}$ that is Brownian. Then, since the branching and motion mechanisms are independent of each other, the many-to-one lemma implies that for $t>0$,
\begin{equation} \label{eq5}
\widetilde{E}_t[\hat{n}_t]= p_t M_t= p_t \lfloor e^{\beta t} g_t\rfloor,
\end{equation} 
where $p_t$ is as before the probability of confinement of a standard BM to $B_t$ over $[0,t]$. It is clear that $\hat{n}_t\leq n_t$. At this point, choose $g$ such that $g_t\geq \gamma_t$ for all $t>0$. Then, using Chebyshev's inequality, it follows from \eqref{eq300}, \eqref{eq4}, and \eqref{eq5} that 
\begin{align}
P(n_t<\gamma_t p_t e^{\beta t})&\leq\widetilde{P}_t\left(\hat{n}_t<\gamma_t p_t e^{\beta t}\right)+g_t  \nonumber \\
&= \widetilde{P}_t\left(\widetilde{E}_t[\hat{n}_t]-\hat{n}_t>\widetilde{E}_t[\hat{n}_t]-\gamma_t p_t e^{\beta t}\right)+g_t   \nonumber \\
&\leq\widetilde{P}_t\left(|\hat{n}_t-\widetilde{E}_t[\hat{n}_t]|>p_t \lfloor e^{\beta t} g_t\rfloor-\gamma_t p_t e^{\beta t}\right)+g_t   \nonumber \\
&\leq \frac{\widetilde{\text{Var}}_t(\hat{n}_t)}{[(g_t-\gamma_t) p_t e^{\beta t}-p_t]^2}+g_t,  \label{eq6}
\end{align}  
where $\widetilde{\text{Var}}_t$ denotes the variance associated to $\widetilde{P}_t$. In the rest of the proof, we estimate $\widetilde{\text{Var}}_t(\hat{n}_t)$. 

Let $\mathcal{P}$ be the probability under which the pair $(i,j)$ is chosen uniformly at random among the $M_t(M_t-1)$ possible pairs in $\mathcal{M}_t$, and let $\mathcal{E}$ be the corresponding expectation. Also, for a generic Brownian motion $X$, let $\text{Var}$ denote its variance, and let $A=\{X(s)\in B_t\:\:\forall\:0\leq s\leq t \}$. Then,
\begin{align} \label{eq7}
\widetilde{\text{Var}}_t(\hat{n}_t)&=\widetilde{\text{Var}}_t\left(\sum_{u\in \mathcal{M}_t} \mathbbm{1}_{A_u}\right) \nonumber \\
&=M_t\text{Var}\left(\mathbbm{1}_{A}\right)+\sum_{1\leq i\neq j\leq M_t}\widetilde{\text{Cov}}_t\left(\mathbbm{1}_{A_i},\mathbbm{1}_{A_j}\right)\nonumber \\
&= M_t(p_t-p_t^2)+M_t(M_t-1)\frac{\sum_{1\leq i\neq j\leq M_t}\widetilde{\text{Cov}}_t\left(\mathbbm{1}_{A_i},\mathbbm{1}_{A_j}\right)}{M_t(M_t-1)} \nonumber \\
&\leq g_t e^{\beta t}(p_t-p_t^2)+g_t e^{\beta t}(g_t e^{\beta t}-1)\left[(\mathcal{E}\otimes \widetilde{P}_t)(A_i\cap A_j)-p_t^2\right],
\end{align}  
where $(\mathcal{E}\otimes \widetilde{P}_t)(A_i\cap A_j)=\mathcal{E}[\widetilde{P}_t(A_i\cap A_j)]$ denotes averaging $\widetilde{P}_t(A_i\cap A_j)$ over the $M_t(M_t-1)$ possible pairs in the randomly chosen set $\mathcal{M}_t$. Let $Q^{(t)}$ be the distribution of the splitting time of the most recent common ancestor of $i$th and $j$th particles under $\mathcal{E}\otimes\widetilde{P}_t$. Applying the Markov property at this splitting time, we obtain
\begin{equation} \label{eq8}
(\mathcal{E}\otimes \widetilde{P}_t)(A_i\cap A_j)=p_t \int_0^t \int_{B_t} p_{t-s,x}^t\, \widetilde{p}^{(t)}(0,s,dx)\,Q^{(t)}(ds),
\end{equation}
where $\widetilde{p}^{(t)}(x,s,dy)$ and $p_{s,x}^t$ are as defined in \eqref{neweq3}. Set $p_s^t=p_{s,0}^t$. Then, it follows from \eqref{neweq4} and \eqref{eq8} that
\begin{equation}\label{eq9}
(\mathcal{E}\otimes \widetilde{P}_t)(A_i\cap A_j)=p_t^2 \int_0^t \frac{1}{p_s^t}\, Q^{(t)}(ds) .
\end{equation}
For $t>0$ define
$$J_t:=\int_0^t \frac{1}{p_s^t}\,Q^{(t)}(ds).$$
Then, by \eqref{eq7} and \eqref{eq9}, we have 
\begin{equation} \label{eq10}
\widetilde{\text{Var}}_t(\hat{n}_t)\leq g_t p_t e^{\beta t}+g_t^2 p_t^2 e^{2\beta t}(J_t-1).
\end{equation}
It is clear that $J_t-1\geq 0$. Next, we bound $J_t-1$ from above.

Recall that $r(t)$ is a distance scale. For $k>0$, we will use $kr(t)$ as a time scale. Note that for large $t$ it is atypical for a BM starting at the origin to escape $B_t=B(0,r(t))$ over $[0,kr(t)]$. For large $t$, split $J_t$ up as 
\begin{equation} 
J_t=\int_0^{kr(t)}\frac{1}{p_s^t}\,Q^{(t)}(ds)+\int_{kr(t)}^t\frac{1}{p_s^t}\,Q^{(t)}(ds), \nonumber
\end{equation}
and define
$$J_t^{(1)}:=\int_0^{kr(t)}\frac{1}{p_s^t}\,Q^{(t)}(ds),\quad \quad J_t^{(2)}:=\int_{kr(t)}^t\frac{1}{p_s^t}\,Q^{(t)}(ds).$$

We first bound $J_t^{(1)}-1$ from above. Observe that $p_s^t$ is nonincreasing in $s$, and estimate
\begin{equation} \label{eq12}
J_t^{(1)}=\int_0^{kr(t)}\frac{1}{p_s^t}\,Q^{(t)}(ds) \leq \frac{1}{p_{kr(t)}^t}.
\end{equation}
From Proposition A,
\begin{equation}\label{eq13}
1-p_{kr(t)}^t=\exp\left[-\frac{r(t)}{2k}(1+o(1))\right].
\end{equation} 
It then follows from \eqref{eq12} and \eqref{eq13} that
\begin{equation}\label{eq14}
J_t^{(1)}-1\leq \frac{\exp\left[-\frac{r(t)}{2k}(1+o(1))\right]}{1-\exp\left[-\frac{r(t)}{2k}(1+o(1))\right]} = \exp\left[-\frac{r(t)}{2k}(1+o(1))\right].
\end{equation} 

To bound $J_t^{(2)}$ from above, we will use the following fact on the distribution $Q^{(t)}$ from \cite[Prop. 5]{E2008}: $Q^{(t)}$ is absolutely continuous with respect to the Lebesgue measure, which we denote by $ds$, and its density function, which we denote by $g^{(t)}$, satisfies 
\begin{equation} \label{eqdensityfunction}
\exists\,C>0,\: s_0>0 \quad \text{such that} \quad \forall\:s\geq s_0,\:\: g^{(t)}(s)\leq C s e^{-\beta s}.
\end{equation}
Since $r(t)=o(\sqrt{t})$ by assumption, this implies that for all large $t$ we have $r(t)\leq t$, which implies $1/p_s^t\leq 1/p_s^{r(t)}$. Here, $p_s^{r(t)}=\mathbf{P}_0(\sigma_{B_{r(t)}}\geq s)$ with $B_{r(t)}=B(0,r(r(t)))$ in accordance with previous notation. Then, since $r(t)\to\infty$ as $t\to\infty$, for all large $t$ and for $kr(t)\leq s\leq t$, 
\begin{equation} 
\frac{1}{p_s^t}\leq \frac{1}{p_s^{r(t)}}=\exp\left[\frac{\lambda_d s}{r^2(r(t))}(1+o(1))\right] \leq \exp\left[\frac{2\lambda_d s}{r^2(r(t))}\right], \nonumber
\end{equation} 
where we have used Proposition B. Then, we continue with
\begin{align}
J_t^{(2)}=\int_{kr(t)}^t\frac{1}{p_s^t}\,Q^{(t)}(ds)&\leq \int_{kr(t)}^t  \exp\left[\frac{2\lambda_d s}{r^2(r(t))}\right]Cse^{-\beta s}ds \nonumber \\
&\leq C \int_{kr(t)}^\infty s \exp\left[-\left(\beta-\frac{2\lambda_d}{r^2(r(t))}\right)s\right] ds \nonumber \\
&\leq \exp\left[-\beta kr(t)(1+o(1))\right], 
\label{eq15}
\end{align} 
where we have used integration by parts. From \eqref{eq14} and \eqref{eq15}, we have
\begin{equation}\label{eq16}
J_t-1=J_t^{(1)}-1+J_t^{(2)}\leq \exp\left[-\frac{r(t)}{2k}(1+o(1))\right]+\exp\left[-\beta kr(t)(1+o(1))\right] .
\end{equation}
To optimize the smallest absolute exponent on the right-hand side of \eqref{eq16}, choose $k$ so that
$$\beta kr(t)=\frac{r(t)}{2k}.$$
This yields $k=\frac{1}{\sqrt{2\beta}}$. With this choice of $k$, we have
\begin{equation} 
J_t-1\leq \exp\left[-\sqrt{\beta/2}\,r(t)(1+o(1))\right].  \nonumber
\end{equation}
It then follows from \eqref{eq6} and \eqref{eq10} that
\begin{equation} \label{eq18}
P(n_t<\gamma_t p_t e^{\beta t})\leq \frac{2g_t}{(g_t-\gamma_t)^2 p_t}e^{-\beta t}+\frac{2g_t^2}{(g_t-\gamma_t)^2}e^{-\sqrt{\beta/2}\,r(t)(1+o(1))}+g_t .
\end{equation}     
By assumption, $\gamma_t=e^{-\kappa r(t)}$ with $\kappa>0$. Choose $g_t=2 \gamma_t$. Then, we can continue \eqref{eq18} with
\begin{equation} \label{eq19}
P(n_t<\gamma_t p_t e^{\beta t}) \leq \frac{4}{\gamma_t p_t}e^{-\beta t}+8 e^{-\sqrt{\beta/2}\,r(t)(1+o(1))}+g_t.
\end{equation}
Using Proposition B, and the assumptions that $r(t)\to\infty$ and $r(t)=o(\sqrt{t})$ as $t\to\infty$, we have
$$\gamma_t p_t= e^{-\kappa r(t)-\frac{\lambda_d t}{r^2(t)}(1+o(1))}=\exp[o(t)].$$
Then, using that $g_t=2\gamma_t=2 e^{-\kappa r(t)}$, it follows from \eqref{eq19} that
\begin{equation*} P(n_t<\gamma_t p_t e^{\beta t})\leq
\begin{cases}
e^{-\kappa\,r(t)(1+o(1))}, &  0<\kappa\leq \sqrt{\beta/2}, \\
e^{-\sqrt{\beta/2}\,r(t)(1+o(1))}, & \kappa>\sqrt{\beta/2}.
\end{cases}  
\end{equation*}
This completes the proof of \eqref{eq200} and the upper bound of \eqref{eq1}. \qed

\section*{Acknowledgements} The author would like to thank the anonymous reviewer for valuable suggestions that helped to significantly improve the presentation of the manuscript.

\bibliographystyle{plain}

\end{document}